\documentclass[draft]{amsart}

 \usepackage{framed,color,pgf}
\colorlet{shadecolor}{gray!30}
\newenvironment{ana}
  {\begin{leftbar}
  \begin{shaded} }
{  \end{shaded}\end{leftbar}}

\newcommand{\arxiv}[1]{\begin{ana} #1\end{ana}}

\renewcommand{\arxiv}[1]{#1}

 \usepackage{tikz}
 \usetikzlibrary{patterns}
 \usepackage[numbers,sort&compress]{natbib}
\usepackage{amsmath,amsthm,amsfonts,amssymb,fancybox,float}
 \usepackage{framed}
 \usepackage{color,soul}
 \usepackage{url}
\newcommand{\comment}[1]{}

\newcommand{\longcomment}[1]{}
\renewcommand{\longcomment}[1]{\ovalbox{\begin{minipage}{.9\textwidth}\color{blue}#1\end{minipage}}}
\renewcommand{\comment}[1]{{\color{blue}\ovalbox{#1}}}

\def\comment#1{}
\def\fddto{\stackrel{\rm f.d.d.}{\Longrightarrow}}
\newcommand{\ind}{{\bf 1}}
\def\indd#1{{\ind}_{\{#1\}}}
\def\inddd#1{{\ind}_{\left\{#1\right\}}}

\newcommand{\proba}{\mathbb P}
\newcommand{\esp}{{\mathbb E}}

\newcommand{\eqnh}{\begin{eqnarray*}}
\newcommand{\eqne}{\end{eqnarray*}}
\newcommand{\eqnhn}{\begin{eqnarray}}
\newcommand{\eqnen}{\end{eqnarray}}
\newcommand{\equh}{\begin{equation}}
\newcommand{\eque}{\end{equation}}

\def\summ#1#2#3{\sum_{#1 = #2}^{#3}}
\def\prodd#1#2#3{\prod_{#1 = #2}^{#3}}

\newcommand{\eqd}{\stackrel{d}{=}}

\def\topp#1{^{(#1)}}

\def\abs#1{\left|#1\right|}
\def\sabs#1{|#1|}
\def\ccbb#1{\left\{#1\right\}}
\def\sccbb#1{\{#1\}}
\def\pp#1{\left(#1\right)}
\def\spp#1{(#1)}
\def\bb#1{\left[#1\right]}

\def\mmid{\;\middle\vert\;}

\def\floor#1{\left\lfloor #1 \right\rfloor}
\def\sfloor#1{\lfloor #1 \rfloor}

\def\aa#1{\left\langle #1\right\rangle}

\def\vv#1{{\boldsymbol #1}}

\def\qmand{\quad\mbox{ and }\quad}

\def\qmwith{\quad\mbox{ with }\quad}
\def\mfa{\mbox{ for all }}

\def\mmas{\mbox{ as }}

\def\wt#1{\widetilde{#1}}

\def\what#1{\widehat{#1}}

\def\limn{\lim_{n\to\infty}}

\def\weakto{\Rightarrow}

\def\Z{{\mathbb Z}}

\def\R{{\mathbb R}}
\def\Rd{{\mathbb R^d}}
\def\N{{\mathbb N}}

\def\B{{\mathbb B}}

\def\qp#1{(#1;q)_\infty}
\def\qps#1{(#1;q)}
\def\qpp#1{\left(#1;q\right)_\infty}

\newcommand{\hidecomment}[1]{}
\newtheorem{theorem}{Theorem}[section]

\newtheorem{lemma}[theorem]{Lemma}
\newtheorem{proposition}[theorem]{Proposition}

\theoremstyle{definition}

\theoremstyle{remark}
\newtheorem{remark}[theorem]{Remark}
\theoremstyle{remark}

\newcommand{\calL}{{\mathcal L}}

\def\<{\langle}
\def\>{\rangle}

\usepackage{dsfont}

\newcommand{\NN}{\mathds{N}}

\newcommand{\heightfunction}{cumulative density}

\newcommand{\eps}{\varepsilon}

\newcommand{\qqq}{g}

\usepackage{graphicx,fancybox,pgf,amsmath,float,amssymb}
\usepackage{fancybox,graphicx,dsfont,pgf}
\usepackage{subfigure,color,wrapfig}

\usepackage{latexsym}

\def\<{\langle}
\def\>{\rangle}

\numberwithin{equation}{section}

\title[Limit fluctuations of open ASEP]{Limit fluctuations for
density of  asymmetric simple exclusion processes  with open boundaries}
\author{W\l odzimierz Bryc}
\address
{
W\l odzimierz Bryc\\
Department of Mathematical Sciences\\
University of Cincinnati\\
2815 Commons Way\\
Cincinnati, OH, 45221-0025, USA.
}
\email{wlodzimierz.bryc@uc.edu}

\author{Yizao Wang}
\address
{
Yizao Wang\\
Department of Mathematical Sciences\\
University of Cincinnati\\
2815 Commons Way\\
Cincinnati, OH, 45221-0025, USA.
}
\email{yizao.wang@uc.edu}

\begin{document}\sloppy
\begin{abstract}We investigate the fluctuations of
\heightfunction\ of
particles in
the asymmetric simple exclusion process
with respect to the stationary distribution (also known as the steady state), as a stochastic process indexed by $[0,1]$. In three phases of the model and their boundaries within the fan region,
we establish a complete picture of the scaling limits of the fluctuations of the density as the number of sites goes to infinity. In the maximal current phase, the limit fluctuation is the sum of two
 independent processes, a Brownian motion and a Brownian excursion. This extends an earlier result by \citet{derrida04asymmetric} for totally asymmetric simple exclusion process in the same phase.
 In the low/high density phases,  the limit fluctuations are Brownian motion. Most interestingly,  at the boundary of the maximal current phase, the limit fluctuation is the sum of two independent
 processes, a Brownian motion and a Brownian meander (or a time-reversal of the latter, depending on the side of the boundary).
Our proofs rely on  a  representation
of the joint   generating function of the %
asymmetric simple exclusion process with respect to the stationary distribution
in terms of joint moments
 of a Markov processes, which is  constructed from
orthogonality measures
 of the
Askey--Wilson
 polynomials.
\end{abstract}

\maketitle

\section{Introduction and main results}

\subsection{Background}

The asymmetric simple exclusion process (ASEP) with open boundaries in one dimension is one of the most widely investigated models for open non-equilibrium systems in the physics literature. %
The process models particles jumping independently with hardcore repulsion over a one-dimensional lattice, which also has particles injected to the left end and removed from the right end, and an external field driving the particles towards the right direction.
The ASEP, despite its simple definition,
captures
representative features of
more complicated models, %
including in particular phase transitions.
The model actually has its origin in modeling protein synthesis in biology \citep{macdonald68kinetics}.
In mathematics literature, the model was first investigated by \citet{spitzer70interaction}, see also \citet[Section 3]{liggett85interacting} for early developments. See more references %
on background, motivations and applications in the survey papers
\citep{derrida06matrix,derrida07nonequilibrium,blythe07nonequilibrium}.

The ASEP with open boundaries is an irreducible %
finite-state Markov process on
the state space $\{0,1\}^n$ with parameters
\equh\label{eq:abcd}
\alpha>0,\quad\beta>0,\quad \gamma\geq 0,\quad \delta\geq 0,\qmand 0\leq q<1.
\eque Informally, the
process models the evolution of the particles located at sites $1,\dots,n$ that can jump to the right with rate 1 and to the left
with rate $q$, if the target site is unoccupied. Furthermore, particles arrive at site $1$  (respectively, $n$),  if empty, at rate $\alpha$  (respectively, $\delta$), and
leave site $n$ (respectively, $1$), if occupied, at rate $\beta$ (respectively, $\gamma$).  The transitions are summarized in Figure \ref{Fig1}.
For $q<1$,   particles move in an asymmetric way,
  with higher rate to the right than to the left;
  in the special case $q = 0$,
particles move only to the right and
the model is known as the totally asymmetric simple exclusion process (TASEP).

We let $\pi_n$ denote the stationary distribution of the ASEP as a Markov process  on $\{0,1\}^n$,
which is also called {\em  the steady state} in the physics literature.
We let $\tau_1,\dots,\tau_n$ denote the occupations of each corresponding location: $\tau_j = 1$ if the $j$-th location is occupied by a particle, and $\tau_j = 0$ otherwise. All statistics of the ASEP are then expressed in terms of $\tau_1,\dots,\tau_n$.

\begin{figure}[H]
  \begin{tikzpicture}[scale=.95]
\draw [fill=black, ultra thick] (.5,1) circle [radius=0.2];
  \draw [ultra thick] (1.5,1) circle [radius=0.2];
\draw [fill=black, ultra thick] (2.5,1) circle [radius=0.2];
  \draw [ultra thick] (5,1) circle [radius=0.2];
   \draw [fill=black, ultra thick] (6,1) circle [radius=0.2];

    \draw [ultra thick] (7,1) circle [radius=0.2];
      \draw [fill=black, ultra thick] (9.5,1) circle [radius=0.2];
   \draw [ultra thick] (10.5,1) circle [radius=0.2];
     \draw[->] (-1,2.3) to [out=-20,in=135] (.5,1.5);
   \node [above right] at (-.2,2) {$\alpha$};
     \draw[->] (10.5,1.5) to [out=45,in=200] (12,2.3);
     \node [above left] at (11.2,2) {$\beta$};
            \node  at (8.25,1) {$\cdots$};  \node  at (3.75,1) {$\cdots$};
      \node [above] at (6.5,1.8) {$1$};
      \draw[->,thick] (6.1,1.5) to [out=45,in=135] (7,1.5);
        \node [above] at (5.5,1.8) {$q$};
            \draw[<-] (5,1.5) to [out=45,in=135] (5.9,1.5);
                 \node [above] at (10,1.8) {$1$};
                \draw[->,thick] (9.6,1.5) to [out=45,in=135] (10.4,1.5);
               \node [above] at (9,1.8) {$q$};
                 \draw[<-] (8.4,1.5) to [out=45,in=135] (9.4,1.5);
       \draw[<-] (-1,-.3) to [out=0,in=-135] (.5,0.6);
   \node [below right] at (-.2,0) {$\gamma$};
    \node [above] at (0.5,0) {$1$};
    \node [above] at (1.5,0) {$2$};
   \node [above] at (2.5,0) {$3$};
     \node [above] at (3.75,0) {$\cdots$};
     \node [above] at (6,0) {\textcolor{white}{$+1$} $k$ \textcolor{white}{$+1$}};
   \node [above] at (5,0) {$k-1$};
       \node [above] at (7,0) {$k+1$};
        \node [above] at (8.25,0) {$\cdots$};
         \node [above] at (10.5,0) {\textcolor{white}{$+$}$n$\textcolor{white}{$1$}};
          \node [above] at (9.5,0) {$n-1$};
        \draw[<-] (10.5,.6) to [out=-45,in=180] (12,-.3);
   \node [below left] at (11.2,0) {$\delta$};
\end{tikzpicture}
\caption{\label{Fig1}Transition rates of the asymmetric simple exclusion process with open boundaries, with parameters $\alpha,\beta,\gamma,\delta, q$.
Black disks represent occupied sites.
}
\end{figure}
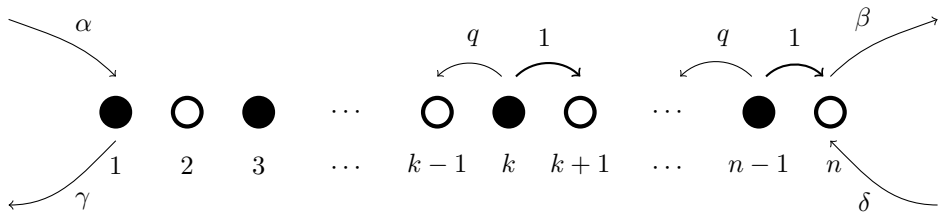

Throughout we assume~\eqref{eq:abcd} and work with the following parameterization of the ASEP, which dates back at least to the 90s in the physics literature (e.g.~\citep{sandow94partially}):
set %
\begin{align*}
\kappa^\pm_{x,y} & = \frac1{2x}\pp{1-q-x+y\pm\sqrt{(1-q-x+y)^2+4xy}},
\end{align*}
and %
denote
 \equh\label{eq:ABCD} A = \kappa^+_{\beta,\delta}, \quad B = \kappa^-_{\beta,\delta}, \quad C=  \kappa^+_{\alpha,\gamma},
\qmand D= \kappa^-_{\alpha,\gamma}. \eque
By definition, $A,C\geq 0$ and it is easy to
check
that $-1<B,D\leq 0$,
 compare~\citep{uchiyama04asymmetric,bryc17asymmetric}.

 The phase transition of the ASEP is known to be
characterized by $A$ and $C$ only. %
For example, it has been known since
\citet{derrida92exact,sandow94partially} that the ASEP has the following three phases:
\begin{enumerate}
\item maximal current phase $A<1,C<1$,
\item  low density phase $C>1,C>A$,
\item  high density phase $A>1,A>C$.
\end{enumerate}
\citet{derrida02exact,derrida03exact} distinguish also the two regions
\begin{enumerate}
\item {fan region} $AC<1$,
\item  {shock region} $AC>1$.
\end{enumerate}
Figure~\ref{Fig3} illustrates the three phases and the two regions.
In this paper, we restrict ourselves to the fan region and its boundary,
 as our
approach does not work for the shock region. See \citep{derrida02exact,derrida03exact} for more discussions
of the properties of ASEP in the shock region.

Another set of commonly used parameters is the pair $(\rho_a,\rho_b)$ with
\[
\rho_a=\frac1{1+C} \qmand \rho_b =\frac A{1+A}.
\]
For example, the fan region and
the
shock region are often characterized equivalently by $\rho_a>\rho_b$ and $\rho_a<\rho_b$, respectively (e.g.~\citep[]{derrida02exact} and \citep[(1.1)]{derrida03exact}).
From the point of view of modeling a non-equilibrium system with open boundaries, the two parameters represent the densities of the two reservoirs connected to the left and the right of the
system. %
For convenience, %
we shall use $A$ and $C$ exclusively in the sequel.

 \begin{figure}%

  \begin{tikzpicture}[scale=1.4]

\draw[scale = 1,domain=6.8:11,smooth,variable=\x,dotted,thick] plot ({\x},{1/((\x-7)*1/3+2/3)*3+5});

\fill[pattern=north east lines, pattern color=gray!60] (5,5)--(5,10) -- plot [domain=6.8:11]  ({\x},{1/((\x-7)*1/3+2/3)*3+5}) -- (11,5) -- cycle;

 \draw[->] (5,5) to (5,10.2);
 \draw[->] (5.,5) to (11.2,5);
   \draw[-, dashed] (5,8) to (8,8);
   \draw[-, dashed] (8,8) to (8,5);
   \draw[-, dashed] (8,8) to (10,10);
   \node [left] at (5,8) {\scriptsize$1$};
   \node[below] at (8,5) {\scriptsize $1$};
     \node [below] at (11,5) {$A = \kappa^+_{\beta,\delta}$};
   \node [left] at (5,10) {$C = \kappa^+_{\alpha,\gamma}$};

\node[above] at (9.5,9.75) {shock region};
\node[above ] at (10.5,7.5) {boundary of  fan/shock region};
\draw[->] (10,7.5) to [out=-45, in = 45] (10,6.8);
 \draw[->] (4.5,8) to [out=-45,in=-135] (6,8);
  \node [above] at (4,8) {$\mathbb B+\widehat{\mathbb B}^{me}$};

  \node[left] at (8,4.5) {$\mathbb B+%
  {\mathbb B}^{me}$};
 \draw[->] (8,4.5) to [out =45,in= -45] (8,6) ;
  \draw[-] (8,4.9) to (8,5.1);
   \draw[-] (4.9,8) to (5.1,8);

 \node [below] at (5,5) {\scriptsize$(0,0)$};
     \node [above] at (4.5,9) {$\mathbb B$};
     \draw[->] (4.5,9) to [out=-45,in=-135] (6,9);
    \node [above] at (6.5,8.5) {LD};
    \node[above] at (8,9) {LD};
      \node [below] at (12,6.5) {$\mathbb B$};
      \draw[->] (12,6.1) to [out=-90,in=-35] (10,5.5);
    \node [below] at (10,6) {HD}; %
      \node [above] at (10,9) {HD};

     \draw[->] (4.5,6.5) to [out=-45,in=-135] (6,6.58);
  \node [above] at (4,6.5) {$\mathbb B+ {\mathbb B}^{ex}$};
 \node [below] at (6.5,6) {MC};%

\end{tikzpicture}

\caption{  \label{Fig3}
A phase diagram for the limit fluctuations of density of the ASEP in the steady state.
LD, HD, MC stand for low density, high density and maximal current %
(phases), respectively.
The fan region consists of the shaded area.
$\B, \B^{ex}$ and $\B^{me}$ stand for Brownian motion, excursion and meander, respectively, and $\what \B^{me}$ stands for the reversed Brownian meander defined by $\{\what \B^{me}_x\}_{x\in[0,1]} \eqd \{ \B^{me}_{1-x} - \B^{me}_1\}_{x\in[0,1]}$. Processes in the sums are assumed to be independent. The multiplicative constants
(possibly depending on $A,C$) in front of processes are omitted.
}
\end{figure}
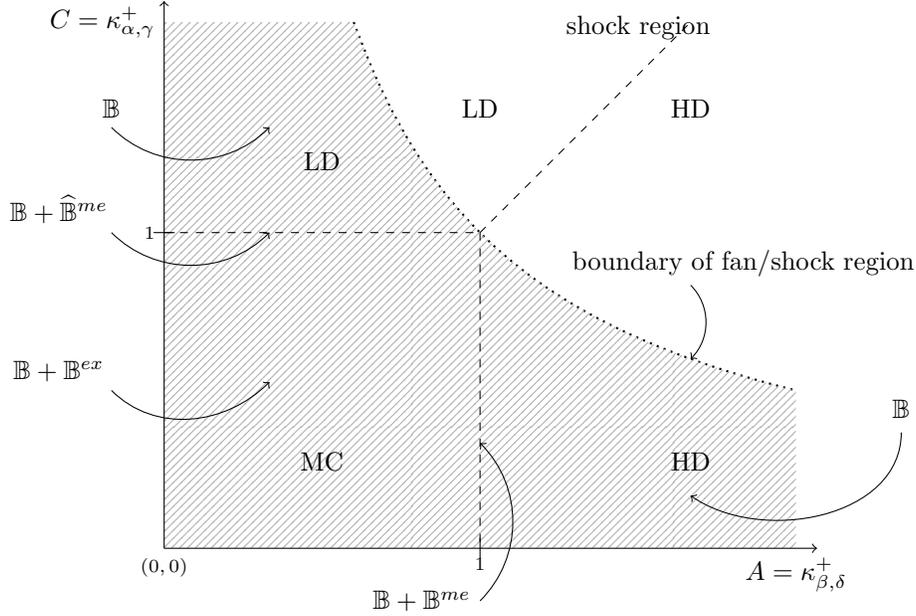

 We are interested in %
  the {\em \heightfunction} function
\[
[0,1]\ni x\mapsto \frac1n \summ j1{\floor{nx}}\tau_j,
\]
which we  consider as a random process under $\pi_n$.
The following limits in probability for the \heightfunction\ function are well known:
  \begin{equation}\label{limH} %
\limn\frac1n\summ j1{\floor {nx}} \tau_j= \left\{\begin{array}{lll}
 \displaystyle  \frac 12x &  A<1,C<1%
& \mbox{(maximal current phase)}
\\ \\
 \displaystyle  %
 \frac1{1+C}x & C>1,C>A%
  & \mbox{(low density phase)}
  \\ \\
 \displaystyle  %
 \frac A{1+A}x & A>1,A>C %
&\mbox{(high density phase)}
 ,
\end{array}
\right.
\end{equation}
see for example \citep{schutz93phase,derrida02exact,uchiyama04asymmetric}.
Phase diagram affects the behavior of many other statistics, including %
current \citep{derrida93exact,sandow94partially}, correlation functions of the density \citep{essler96representations,uchiyama05correlation}, and the large deviation functionals of the density \citep{derrida02exact}  or the current \citep{degier11large}.
See \citep{derrida07nonequilibrium} and more references therein.

The fluctuations of the %
\heightfunction\  function  %
with appropriate normalization %
are
easy to
describe
for the boundary of the fan region ($AC = 1$). Since it is known that in this case, %
 $\tau_1,\dots,\tau_n$ are i.i.d.~Bernoulli random variables with %
 mean
$A/(1+A) = 1/(1+C)$ (see \citep{enaud04large} and Remark~\ref{rem:T0}), the scaling limit
is
the Brownian motion, an immediate consequence of Donsker's theorem %
\citep{billingsley99convergence}.
\begin{theorem}[Boundary of fan region]%
\label{T0}
When $AC = 1$,
$$
\frac1{\sqrt n}\ccbb{\summ j1{\floor {nx}}\pp{\tau_j - \frac A{1+A}}}_{x\in[0,1]} \Rightarrow \frac{\sqrt{A}}{1+A}\ccbb{\B_x}_{x\in[0,1]}
$$
as $n\to\infty$ in the space $D([0,1])$.
\end{theorem}
One might expects  naturally
 that the Brownian-motion behavior at the boundary of the fan region  persists if one looks at the phase that is close to the boundary.
However,
this intuition is not entirely correct, as the limit is %
non-Gaussian
 for
 $A<1,C<1$ arbitrarily close to the point $(A,C) = (1,1)$ %
 at the boundary, a remarkable result
due to
\citet{derrida04asymmetric}, who showed %
this
 in the special case $q=0, \gamma=\delta = 0$.

\subsection{Main results}

In this paper, we provide a complete picture of the  limit fluctuations of the \heightfunction\  function,
that is,
 of
the process
$\{\summ j1{\floor {nx}} \tau_j \}_{x\in[0,1]}$
with appropriate normalization as $n\to\infty$, in the fan region.
First,
 as conjectured in
 \citep[Section 3]{derrida04asymmetric},
 we show that %
  the limit fluctuation
 for a  full range of parameters \eqref{eq:abcd}   in the maximal current phase
 is    the same as for
  the case of $q=0,\gamma=\delta = 0$ studied in  \citep{derrida04asymmetric}.
Second,
and most interestingly,
we identify two different limit fluctuations at the boundary of the maximal current phase.
 Third,
 in the low/high density phases in the fan region, we show that
 the scaling limit of fluctuations is a Brownian motion.

Our results are stated in terms of Brownian motion, Brownian excursion and Brownian meander,  denoted by $\B$, $\B^{ex}$ and $\B^{me}$ respectively throughout this paper.
One may  think of  Brownian excursion as the Brownian bridge conditioned to stay strictly positive until time $t=1$, and Brownian meander as the Brownian motion conditioned to
stay strictly positive over time interval $(0,1]$.
See for example %
\citep{revuz99continuous,pitman06combinatorial,pitman99brownian,durrett77weak,janson07brownian,yen13local} for more background and applications.

We first state our results on the maximal current phase and its boundary.
Introduce
\[
h_n(x) = \summ j1{\floor {nx}}\pp{\tau_j-\frac12}, x\in[0,1],
\]%
and view $\{h_n(x)\}_{x\in[0,1]}$ as a
stochastic process
  with law induced by $\pi_n$.
The following theorem  extends the already mentioned result of \citet{derrida04asymmetric} to %
 a larger range of parameters \eqref{eq:abcd} %
  confirming the conjecture in \cite[Section 3]{derrida04asymmetric}. We let `$\fddto$' denote convergence of
finite-dimensional distributions. Recall that
definition~\eqref{eq:ABCD} gives $A\ge 0, C\ge 0$.

\begin{theorem}[Maximal current phase]\label{T1}
If $A<1,C<1$ then
\[
\frac1{\sqrt n}\ccbb{h_n(x)}_{x\in[0,1]}\fddto
\frac1{2\sqrt 2}\ccbb{\B_{x} + \B^{ex}_{x}}_{x\in[0,1]},
\]
as $n\to\infty$,
where the Brownian motion $\B$ and the Brownian excursion $\B^{ex}$ are
 independent %
stochastic processes.
\end{theorem}
The boundary of the maximal current phase, see
Figure \ref{Fig3}, splits into three regions
with
different %
limit
 fluctuations: the corner  point where  $A=1,C=1$
with asymptotically Brownian fluctuations described in Theorem \ref{T0}, and
two line-segments  corresponding to $A<1, C=1$    and $A=1, C<1$ with the following fluctuations.
\begin{theorem}[Boundary of maximal current phase]\label{T1'}
We have,
\[
\frac1{\sqrt n}\ccbb{h_n(x)}_{x\in[0,1]}\stackrel{\rm f.d.d.}{\Longrightarrow}
\left\{
\begin{array}{ll}
\displaystyle \frac1{2\sqrt 2}\ccbb{\B_{x} + \B^{me}_{x}}_{x\in[0,1]} & A=1, C<1 \\  \\
  \displaystyle \frac1{2\sqrt 2}\ccbb{\B_{x} + \B^{me}_{1-x}-\B^{me}_{1}}_{x\in[0,1]} & A<1, C=1  \\ %
\end{array}
\right.
\]
as $n\to\infty$,
where the Brownian motion $\B$ and the Brownian meander  $\B^{me}$ are
independent %
stochastic processes.
\end{theorem}

For the low/high density phases,
 we  use centering %
 as indicated in \eqref{limH}.
 For $x\in[0,1]$,
 introduce %
    \begin{align}
h_n^{\rm L}(x) & =\summ j1{\floor{nx}} \pp{\tau_j-\frac{1}{1+C}}\nonumber\\
    h_n^{\rm H}(x) & =\summ j1{\floor{nx}}\pp{\tau_j-\frac{A}{1+A}},\label{eq:h_HD}
    \end{align}
    and view both as stochastic processes   with laws induced by $\pi_n$.
 \begin{theorem}[Low/high density phases of fan region]\label{T2}
Suppose $AC<1$.
In the low density phase,  %
$C>1$, we have
 \begin{equation*}%
    \frac{1}{\sqrt{n}}\ccbb{h^{\rm L}_n(x)}_
    {x\in[0,1]}
    \stackrel{\rm f.d.d.}
  {\Longrightarrow}  %
    \frac{\sqrt{C}}{1+C}
\ccbb{\B_x}_
{x\in[0,1]}
 \mbox{ as $n\to\infty$}.
  \end{equation*}
In the high density phase, %
$A>1$, we have
  \begin{equation*}
   \frac{1}{\sqrt{n}} \ccbb{ h^{\rm H}_n(x)}_
   {x\in[0,1]}
   \stackrel{\rm f.d.d.}
  {\Longrightarrow}  %
  \frac{\sqrt{A}}{1+A}
  \ccbb{\B_x}_
{x\in[0,1]}
  \mbox{ as $n\to\infty$}.
  \end{equation*}
 \end{theorem}

The paper is organized as follows. In Section \ref{sec:overview} we describe informally
the basic ideas behind the proof. Section~\ref{sec:prelim}
provides technical background on Askey--Wilson processes and generating functions of ASEP.
In Section~\ref{sec:T2} we prove Theorem~\ref{T2}. Section~\ref{sec:T1} presents %
proofs of
Theorems~\ref{T1} and~\ref{T1'}. In Appendix \ref{appendix:Laplace} we discuss the Laplace transform criterion for
  weak convergence that we use.

\subsection{Overview of the  proof}
\label{sec:overview}
Our starting point is the identity
\equh\label{eq:BW0}
\aa{\prod_{j=1}^n t_j^{\tau_j}}_n=\frac{\esp\bb{\prod_{j=1}^n(1+t_j+2\sqrt{t_j} Y_{t_j})}}{2^n\esp(1+Y_1)^n}, \mfa 0<t_1\le t_2\le\cdots\le t_n,
\eque
which expresses the probability generating function of ASEP on the left-hand side  as a functional of an auxiliary Markov process $\{Y_t\}_{t\ge 0}$.
The process $\{Y_t\}_{t\ge 0}$, introduced in
\citep{bryc10askey}, is an inhomogeneous Markov process
 with transition probabilities constructed
 form the Askey--Wilson laws, that is, from the ``weight functions" of the Askey--Wilson polynomials \citep{askey85some}, as  described in Section \ref{sec:prelim-AW}.
   The parameters $A,B,C,D$ introduced in~\eqref{eq:ABCD}  are the parameters of
this process and our notation here is consistent with \citep{bryc10askey,bryc17asymmetric}.
Identity~\eqref{eq:BW0} comes from  \citep{bryc17asymmetric} and is a new representation of the matrix ansatz,
which  is a powerful and commonly used method  developed in the %
seminal work of \citet{derrida93exact}.
Our approach, however, is of an analytical nature that is different from most applications of the matrix ansatz to the ASEP in the literature (see Remark \ref{rem:comparison}).

Theorems~\ref{T1},~\ref{T1'} and \ref{T2}  %
are
established by
representing
the Laplace transforms of %
the
 finite-dimensional distributions of
processes %
$h_n$, %
and $h^{\rm H}_n$,
 normalized by $\sqrt{n}$, in terms of this auxiliary Markov process  $\{Y_t\}_{t\ge 0}$.
 In this representation, which is a straightforward application of \eqref{eq:BW0},  see
  \eqref{eq:Lap-EZ} and \eqref{eq:Laplace}, the arguments of the Laplace transform become  time arguments for the Markov process. This reduces the
  study of fluctuations of ASEP as the system size $n$ increases  to the analysis of asymptotic behavior of   Markov process   $\{Y_t\}_{t\ge 0}$ near $t=1$.
  In the case of the low/high density regimes, this then leads to a quick proof for the limit fluctuations (Theorem~\ref{T2}).
 The proof for the maximal current phase and its boundary is  more involved  and   requires two %
 additional
 ingredients that %
 we now explain.

   The first ingredient is  the so-called {\em tangent process} \cite{falconer03local} at the
   upper
   boundary of
   the
   support
   of  process $\{Y_t\}_{t\ge 0}$.
   The tangent process, denoted by %
   $\{\Z_t\}_{t\ge0}$,
    is a positive $1/2$-self-similar Markov process with explicit transition probability density function (see Section \ref{sec:Z})
and arises as follows.
Intuitively, the tangent process captures the asymptotic fluctuations of the process $\{Y_t\}_{t\ge 0}$, as the time parameter $t$ is approaching 1 and $Y_1$ is approaching the upper boundary end of the support $[-1,1]$.
To utilize this concept,
 we introduce a sequence of  Markov processes
$\{\what Y\topp n_s\}_{s\geq 0}$  which up to a multiplicative constant behave roughly like  $\{(1-Y_{1-\eps
s})/\eps^2\}_{s\geq 0}$ for $\eps^2\sim 1/n$, (for precise definition, see \eqref{eq:Yhat} below). In Proposition
\ref{prop:tangent} we   show that as $\eps\to 0$ we have \equh\label{eq:tangent0} \mathcal{L}\pp{\ccbb{\what Y_{s}\topp
n}_{s\ge 0}\mmid \what Y_0\topp n = u} \fddto \mathcal{L}\pp{\ccbb{\Z_{s}}_{s\ge 0}\mmid \Z_0 = u}, \mfa u>0. \eque
The left-hand side is interpreted as the law of the Markov process $\{\what Y_s\topp n\}_{s\ge 0}$ starting at $\what Y_0\topp n = u$;  similar interpretation applies for the right-hand side.
In the second part of Proposition \ref{prop:tangent} we show that
under appropriate normalization, the density of $\what Y_{0}\topp n$ converges to an infinite measure
$\nu(du)$ which is proportional to either $u^{1/2}du $ (in the case $A<1, C<1$) or  $u^{-1/2}du $ (in the case $A=1, C<1$). %

Up to a normalizing constant, the Laplace transform   of the
  finite-dimensional distributions of $\{h_n(x)\}_{x\in[0,1]}$
 takes the form of
$\esp G_n\spp{\what Y_{s_1}\topp n,\dots,\what Y_{s_{d}}\topp n ,\what Y_0}$ for some $s_1>\cdots>s_d>0$, where the
 sequence of functions $G_n:\R^{d+1}\to \R_+$, converges  to function $G$ defined in \eqref{eq:LimG_n}.
 Convergence in \eqref{eq:tangent0} is a key step  to show that these  Laplace transforms converge to the limit %
   given by the functional
$$ \int_{\R_+^{}}\esp( G(\Z_{s_1},\dots,\Z_{s_{d}},u)\mid \Z_0=u)\nu(du)
$$
of the tangent process $\Z$, see \eqref{lim-phi}.

 The second ingredient of our proof consists of some recently developed duality formulas \citep{bryc18dual} that express the Laplace transforms of Brownian excursion
  and meander in terms of the tangent process %
 $\{\Z_t\}_{t\ge 0}$ (\eqref{eq:ex} and \eqref{eq:me}).
We recognize that the integral above has   two factors:  the Laplace transform of  the Brownian motion, and
a functional of $\{\Z_s\}_{s\ge 0}$
which we  identify as the  Laplace transform of
Brownian excursion ($A<1,C<1$), see  \eqref{eq:EG}, or of
Brownian meander ($A=1,C<1$), see  \eqref{eq:EG2}.
 A delicate issue actually arises here as due to the use of Markov process  we establish convergence of the Laplace transforms  only in an open region away from the origin in $\Rd$.
We clarify  how this leads to the desired weak convergence   in Appendix \ref{appendix:Laplace}.

Technical difficulties arise in  the above approach  when transitions probabilities of $\{Y_t\}_{t\ge 0}$
are of mixed type near $t=1$.
 We avoid this issue
by applying the so-called
particle-hole duality,
which is a well known symmetry feature of the ASEP.
In particular, the case $A<1, C=1$ and the low density phase will be derived from the case $A=1, C<1$ and the high density phase, respectively, by this duality.

\begin{remark}
   In principle, %
   our
   approach might work for the weakly asymmetric exclusion process, as in  \cite{derrida05fluctuations}, where the authors consider
 the case $q\uparrow1$ at a rate that may depend on $n\to\infty$  and show  that the fluctuations are Gaussian.
 This would require to determine first the relevant tangent process as $q\uparrow1$.

We also mention that there is a huge literature on  the  asymptotic behavior  of ASEP as a temporal-spatial process, by letting the ASEP to
evolve from
a non-stationary distribution, and %
possibly with $q\uparrow 1$ %
at the rate that may depend on $n\to\infty$.
See for example
\citep{eyink90hydrodynamics,eyink91lattice,goncalves17nonequilibrium,corwin16open}
and references therein.
Such results are beyond the scope of  %
our methods.
\end{remark}

\section{Askey--Wilson process and ASEP}\label{sec:prelim}
\subsection{Askey--Wilson process}\label{sec:prelim-AW}

Askey--Wilson processes are a family of Markov processes based on Askey--Wilson measures, which we recall first.
The Askey--Wilson measures are the probability measures that make  the
Askey--Wilson polynomials orthogonal.
We do not use these polynomials here, and instead we write directly the orthogonality measure
 as
given in
\cite{askey85some},
see also \citep[Section 3.1]{koekoek96askey}
where a typo to   weight of higher atoms is corrected.
The formulas below incorporate this correction and probabilistic normalization, and come from \cite{bryc10askey}.

The Askey--Wilson probability measure $\nu(dy;a,b,c,d,q)$ depends on %
 five
parameters $a,b,c,d,q$.
It is assumed that $q\in(-1,1)$. %
For the   parameters $a,b,c,d$, it is assumed that they are all real, or two of the parameters are real and the other two form a complex conjugate pair,
or the parameters form two complex conjugate pairs,
and in addition
\begin{equation}\label{eq:restriction}
ac,ad,bc,bd,qac,qad,qbc,qbd, abcd, qabcd \not \in [1,\infty).
\end{equation}
The Askey--Wilson measure is invariant with respect to permutations of %
$a,b,c,d$.
More precisely,
the measure is of mixed type
$$
\nu(dy;a,b,c,d,q)=f(y;a,b,c,d,q)dy+\sum_{z\in F(a,b,c,d,q)}\,p(z)\delta_z(dy),
$$
with the absolutely continuous part supported on $[-1,1]$ and with the discrete part supported on a finite or empty set $F$.
For certain choices of parameters, the measure can be only discrete or only absolutely continuous.
The absolutely continuous part is %
\begin{equation}\label{eq:f}
  f(y;a,b,c,d,q)=\frac{\qp{q,ab,ac,ad,bc,bd,cd}}{2\pi\qp{abcd}\sqrt{1-y^2}}\,\left|\frac{\qp{e^{2i\theta_y}}}
{\qp{ae^{i\theta_y},be^{i\theta_y},ce^{i\theta_y},de^{i\theta_y}}}\right|^2,
\end{equation}
where $y=\cos\,\theta_y$
(with the convention that $f(y;a,b,c,d,q)=0$ when $|y|>1$).
Here and below, for complex $\alpha$,  $n\in\N\cup\{\infty\}$ and $|q|<1$ we
use the
$q$-Pochhammer symbol
\begin{equation}
  \label{eq:Pochhamer}
\qps{\alpha}_n=\prod_{j=0}^{n-1}\,(1-\alpha q^j), \quad \qps{a_1,\cdots,a_k}_n =\prodd j1k\qps{a_j}_n.
\end{equation}
The set $F=F(a,b,c,d,q)$  of atoms of $\nu(dy;a,b,c,d,q)$ is   non-empty if  there is a parameter $\alpha\in\{a,b,c,d\}$  with
$|\alpha|> 1$. In this case, necessarily $\alpha$ is real and generates atoms: for example,
 if $|a|> 1$ then it generates the atoms
\begin{equation}
  \label{eq:atoms}
  y_j=\frac12\pp{aq^j+\frac1{aq^j}} \mbox{ for $j=0,1,\dots$ such that $|aq^j|\ge 1$},
\end{equation}
and the corresponding masses are
\begin{align}
 p(y_0;a,b,c,d,q) & =\frac{\qp{a^{-2},bc,bd,cd}}{\qp{b/a,c/a,d/a,abcd}},\; \label{eq:p0}\\
p(y_j;a,b,c,d,q) & =p(y_0;a,b,c,d,q)\frac{\qps{a^2,ab,ac,ad}_j\,(1-a^2q^{2j})}{\qps{q,qa/b,qa/c,qa/d}_j(1-a^2)}\left(\frac{q}{abcd}\right)^j,
j\ge 1.\nonumber
\end{align}
The formula of $p(y_j;a,b,c,d,q)$ given here only applies for $a,b,c,d\ne 0$, and takes a different form otherwise. We  shall however  only need $p(y_0;a,b,c,d,q)$ in this paper.

The Askey--Wilson process is a time-inhomogeneous Markov process introduced in \citet{bryc10askey}, based on Askey--Wilson measures. It is then explained in~\citet{bryc17asymmetric} how each ASEP with parameters $\alpha,\beta>0, \gamma,\delta\ge 0, q\in[0,1)$ is associated to an Askey--Wilson process $Y$, the parameters of which are denoted by $A,B,C,D,q$, with $A,B,C,D$ given in~\eqref{eq:ABCD}.

As we already noted, \eqref{eq:ABCD} implies $A,C\geq 0$ and $-1<B,D\leq 0$.
So for the Askey--Wilson process to exist, the restriction~\eqref{eq:restriction} becomes $AC<1$, which we assume throughout in the sequel. Then, the Askey--Wilson process with parameters $(A,B,C,D,q)$ is introduced as the Markov process with marginal  distribution
\begin{equation*}%
\proba(Y_t\in dy)=\nu\pp{dy;A\sqrt{t},B\sqrt t,C/\sqrt{t},D/\sqrt t,q},\quad 0<t<\infty,
\end{equation*}  and the transition probabilities
\begin{equation}
\label{Y-transitions}
\proba(Y_t\in dz\mid Y_s=y)=\nu\pp{dz;A\sqrt{t},B\sqrt t,\sqrt{s/t}(y+\sqrt{y^2-1}),\sqrt{s/t}(y-\sqrt{y^2-1})},
\end{equation}
for $0<s<t$, $y,z>0$.
When $|y|<1$, $y\pm \sqrt{y^2-1}$ is understood as $e^{\pm i\theta_y}$ with $\theta_y$ determined by $\cos\theta_y = y$. It was shown in \citep{bryc10askey} that the above marginal and transition laws are consistent and determine a Markov process indexed by $t\in[0,\infty)$.
The Askey--Wilson process turned out to be closely related  to a large family of Markov processes,  the so-called quadratic harnesses \citep{bryc07quadratic} in the literature; see \citep[Section 1.3]{bryc17asymmetric} for more on this connection.
More explicit expressions for the law of $Y$ will  appear below when they are needed in the proofs.

\subsection{Generating function of ASEP via Askey--Wilson process}
Let $\aa \cdot _n$ denote the expectation with respect to
the invariant measure  $\pi_n$ of ASEP.
\citet{derrida93exact} derives the well known  matrix ansatz method that
 provides    an explicit expression of the joint generating function, which made many calculations of the model possible. Formally,
  for any $t_1,\dots,t_n>0$,
 from \citep{derrida93exact}
 one can write %
\begin{equation}\label{MatrixAnsatz}
\aa{\prodd j1n t_j^{\tau_j}}_n = \frac{\langle W|(\mathsf{E}+t_1\mathsf{D})\times\cdots\times(\mathsf{E}+t_n\mathsf{D})|V\rangle}{\langle W|(\mathsf{E}+\mathsf{D})^n|V\rangle},
\end{equation}
for a pair of infinite matrices $\mathsf{D},\mathsf{E}$, a row vector $W$ and a column vector $V$, satisfying
\begin{align*}
\mathsf {DE} - q\mathsf{ED} & = \mathsf {D} + \mathsf {E},\\
\langle W|(\alpha \mathsf E - \gamma \mathsf D) & = \langle W|,\\
(\beta\mathsf D - \delta\mathsf E)|V\rangle & = |V\rangle.
\end{align*}
See \citep{derrida06matrix,derrida07nonequilibrium} for reviews of literature. However, for our purpose, we shall apply an alternative expression developed recently
in \citep[Theorem 1]{bryc17asymmetric}, summarized
in
 the following theorem.
\begin{theorem}\label{T-BW17}
Consider the parameterization $A,B,C,D$ in~\eqref{eq:ABCD} for an ASEP with parameters $\alpha,\beta>0,\gamma,\delta\ge 0$.
Suppose that $AC<1$ and $q\in[0,1)$. Then for $0<t_1\leq t_2\leq \dots\leq t_n$, the joint generating function of
the stationary distribution of the ASEP with parameters is
  \begin{equation}
 \label{eq:BW}
\aa{\prod_{j=1}^n t_j^{\tau_j}}_n=\frac{\esp\bb{\prod_{j=1}^n(1+t_j+2\sqrt{t_j}Y_{t_j})}}{2^n\esp(1+Y_1)^n},
  \end{equation}
  where $\{Y_t\}_{t\ge 0}$ is the Askey--Wilson process with parameters $(A,B,C,D,q)$.
\end{theorem}

Now, to establish our limit theorems, it suffices to analyze the asymptotics of the two expectations
that appear in the numerator and in the denominator
 on the right-hand side of~\eqref{eq:BW}.
For this purpose, we shall see that asymptotically, only the law of
$\{Y_t\}_{t\in[1-\varepsilon,1]}$ matters for arbitrarily small $\varepsilon>0$.
 We first proceed
 in Section~\ref{sec:T2} with
 the proof of the low/high density phases, in which case the law of $Y_t$ near upper boundary
of its support
 is easy to analyze.

\begin{remark}\label{rem:AW}
The connection between Askey--Wilson polynomials and the ASEP has been known for a long time, see for example \citep{sasamoto99one,uchiyama05correlation,uchiyama04asymmetric}.
In \citep{uchiyama04asymmetric,uchiyama05correlation}, using Askey--Wilson polynomials and complex integrals, the asymptotics of most commonly investigated statistics are computed,
including current, density, partition function and the multiple-point correlation function, for results in both fan and shock regions (except the case $A = C >1$
where the steady state does not have constant density).
The identification %
of the Askey--Wilson Markov process in Theorem \ref{T-BW17}
turned out to be convenient for our proofs,
at the expense of  restriction of parameters of ASEP to the fan region $AC<1$. Notice that in general, Askey--Wilson polynomials do not necessarily admit a
positive
orthogonality measure,
and conditions on the coefficients for %
its
 existence are subtle (see \citep{bryc10askey}).
\end{remark}

\begin{remark}\label{rem:comparison}
The version of the matrix ansatz method that we use is more analytic so our method differs from the usual applications of the matrix ansatz that seem to have more combinatorial flavor. For example,   a formula for the joint distribution of the increments of $h_n$  is given in \citep[Eq.~(3.7)]{enaud04large} and used to derive the large deviation principle via a combinatorial argument \citep[Eq.~(3.16), (3.17)]{enaud04large}, essentially by expressing the probability of interest as a sum of probabilities indexed by different paths and then counting the number of paths that asymptotically have the same order of probabilities. This argument is of a completely different nature of ours.

The combinatorial nature of the matrix ansatz method has also been exploited in applications to problems on combinatorial enumeration \citep{corteel2011matrix}.
\end{remark}
\begin{remark}\label{rem:T0}
At the boundary of the fan region,  $AC=1$, one can read from \citep[Eq.~(2.14) and (2.15)]{bryc10askey} that
 \[
 Y_t = \frac{[A+B-AB(C+D)]t+C+D-CD(A+B)}{2\sqrt t(1-ABCD)} = \frac12\pp{A\sqrt t+\frac1{A\sqrt t}},
 \]
 is deterministic.
 Now,
 from \eqref{eq:BW} we can read out that  $\{\tau_j\}_{j=1,\dots,n}$ are independent, and
 $\aa{t_j^{\tau_j}}_n = t_jA/(1+A) + 1/(1+A)$.
So these are Bernoulli random variables with
  $$\aa{\tau_j}_n=\frac A{1+A}.$$
 Theorem~\ref{T0} now is a consequence of the well known Donsker's theorem
 \cite{billingsley99convergence}.

\end{remark}
\section{
Proofs
 for
low/high density phases}\label{sec:T2} In this section, we investigate the case $A>1, AC<1$ and $C>1, AC<1$. In the
representation~\eqref{eq:BW}, the law of the associated Askey--Wilson process with parameters $(A,B,C,D;q)$ may have atoms. It
turns out that we will only need the point mass on the largest atom. We shall only use this representation for the high density
phase ($A>1, AC<1$). For the low density phase, the result shall follow by  the particle-hole duality.

Fix $A>1$ and $C<1/A$. %
Recall that $B,D\in(-1,0]$ and atoms are only generated by parameters that have absolute value larger than 1, so possibly by
$A\sqrt t, C/\sqrt t$ and $D/\sqrt t$. When $t\in(\max\{1/A^2,D^2\},1]$,  all the atoms are generated by $A\sqrt t$
by~\eqref{eq:atoms} with $a = A\sqrt t$, and in this case we let $y_j(t)$ denote the $(j+1)$-th largest atom of the law of
$Y_t$. In particular, we have
 \[
y_0(t) = \frac12\pp{A\sqrt t+\frac1{A\sqrt t}}>1
 \mbox{ for } t\in\left(\max\ccbb{\frac1{A^2},D^2},1\right].
\]
We shall need the mass of $Y_1$ on $y_0(1)$, which is denoted by, recalling~\eqref{eq:p0},
\[
\mathfrak p_0 = p(y_0(1);A,B,C,D,q) = \frac{\qp{1/A^2,BC,BD,CD}}{\qp{B/A,C/A,D/A,ABCD}}.
\]
In the sequel, we write $a_n\sim b_n$ if $\limn a_n/b_n = 1$.
\subsection{Proof
of Theorem \ref{T2}
for the high density phase $A>1, AC<1$}
We
prove
the convergence of corresponding Laplace transform. We first recall
the Laplace transform of the finite-dimensional distribution of
the
Brownian motion. %
For $x_0=0<x_1<\dots<x_d\leq  x_{d+1}=1$,
 $c_1,\dots,c_d>0$, $s_k = c_k+\cdots+c_d$, $k=1,\dots,d$, and $s_{d+1} = 0$, we have
\begin{multline}\label{eq:Brownian}
\esp\exp\pp{-\summ k1d {c_k}\B_{x_k}}   = \esp\exp\pp{-\summ k1d (s_k-s_{k+1})\B_{x_k}}\\
=
 \esp\exp\pp{-\summ k1d {s_k}(\B_{x_k}-\B_{x_{k-1}})} = \exp\pp{\frac12\summ k1{d+1}s_k^2(x_k-x_{k-1})} .
\end{multline}

For the ASEP in the high density phase, consider the centered \heightfunction\
  function
  \eqref{eq:h_HD}
and its   Laplace transform %
with argument
$\vv c= (c_1,\dots,c_d)\in\R_+^d$
defined by
\[
\varphi^{\rm H}_{\vv x,n}(\vv c)  = \aa{ \exp\pp{-\summ k1d c_kh^{\rm H}_n(x_k)}}_n.
\]
Note that in Theorem~\ref{T2}, the limit Brownian motion is scaled by
$%
{\sqrt A}/(1+A)$.
Therefore,
by Theorem \ref{Thm-Laplace}
to prove the high density phase of Theorem~\ref{T2}
it suffices to prove
\[
\limn\varphi^{\rm H}_{\vv x,n}\pp{\frac{\vv c}{\sqrt n}} = \exp\pp{\summ k1{d+1}\frac{A}{2(1+A)^2}s_k^2(x_k-x_{k-1})}.%
\]
To do so, we first write
\begin{align*}
\varphi^{\rm H}_{\vv x,n}(\vv c) & =  \aa{\exp\pp{-\summ k1d \sum_{j=\floor{n x_{k-1}}+1}^{\floor{nx_k}}\pp{\tau_j-\frac A{1+A}}(c_k+\cdots+c_d)}}_n\\
& = \exp\pp{\summ k1d \frac {A}{1+A}s_k(n_k-n_{k-1})}
\aa{{\prodd k1{d+1}\prod_{j=n_{k-1}+1}^{n_k}(e^{-s_k})^{\tau_j}}}_n,
\end{align*}
where
 $n_k = \floor{nx_k}, k=1,\dots,d+1$.
By \eqref{eq:BW},
\begin{equation}
  \label{eq:Lap-EZ}
  \varphi^{\rm H}_{\vv x,n}(\vv c) = \frac1{2^nZ_n}\esp\bb{\prodd k1{d+1}\pp{\frac{1+e^{-s_k}+2e^{-s_k/2}Y_{e^{-s_k}}}{e^{-s_kA/(1+A)}}}^{n_k-n_{k-1}}}
\end{equation}
with $Z_n = \esp(1+Y_1)^n$.

\begin{lemma}
  If    $A>1$ and $AC<1$, then
    \begin{equation*}%
      Z_n\sim \frac{(1+A)^{2n}}{2^nA^n}\mathfrak p_0.
    \end{equation*}

\end{lemma}
This result has been known in the literature. See Remark~\ref{rem:Z_n}. We provide a proof here for
completeness.
\begin{proof}
Let $y_1^*(1) = \max(y_1(1),1)$ denote the upper bound of the support of the law of $Y_1$ on $\R\setminus\{y_0(1)\}$. Note that
$y_1^*(1)<y_0(1)$
(recall that $y_0(1)$ and $y_1(1)$ are the
locations of the
point masses of the first and second atom of $Y_1$,
see \eqref{eq:atoms}, and that the first atom lies above $1$ and above the second atom). %
 It then follows that
$$
Z_n= \int_{\{y_0(1)\}}(1+y)^n\nu(dy;A,B,C,D,q) + \int_{[-1,y_1^*(1)]}(1+y)^n\nu(dy;A,B,C,D,q).
$$
The first term equals
\[
\mathfrak p_0(1+y_0(1))^n = \mathfrak  p_0\frac{(1+A)^{2n}}{2^nA^n},
\] and the second term is bounded from above by $(1+y_1^*(1))^n$, which
converges to 0 when divided by  $(1+y_0(1))^n$.
\end{proof}
In view of the asymptotics of $Z_n$,
we introduce
\[
\psi(s,y) = \frac{1+e^{-s}+2e^{-s/2}y}{e^{-sA/(1+A)}}\frac A{(1+A)^2},
\]
and have
\[
\varphi^{\rm H}_{\vv x,n}\pp{\frac{\vv c}{\sqrt n}} \sim \frac{M_n}{\mathfrak p_0} %
\qmwith M_n = \esp\bb{\prodd k1{d+1}\psi\pp{\frac{s_k}{\sqrt n},Y_{e^{-s_k/\sqrt n}}}^{n_k-n_{k-1}}}.
\]
The desired result now follows from the following.
\begin{lemma}With the notation above,
\[
\limn M_n = \mathfrak p_0 \exp\pp{\summ k1d\frac{A}{2(1+A)^2}s_k^2 %
(x_k-x_{k-1})}.
\]
\end{lemma}
\begin{proof}We have that
\equh\label{eq:psi}
\psi(s,y_0(e^{-s})) = \pp{\frac 1{1+A}e^{sA/(1+A)}+\frac A{1+A}e^{-s/(1+A)}} = 1+\frac{As^2}{2(1+A)^2}+o(s^2)
\eque
as $s\downarrow 0$. Introduce $s_{k,n} = s_k/\sqrt n$ and $t_{k,n} = e^{-s_k/\sqrt n}$.
Note that due to our choice of $s_k$, we have $t_{1,n}<t_{2,n}<\dots<t_{d,n}<1=t_{d+1,n}$.

We write $M_n = M_{n,1}+M_{n,2}$ with
\[
 M_{n,1} = \esp\bb{\prodd k1{d+1}\psi\pp{s_{k,n},Y_{t_{k,n}}}^{n_k-n_{k-1}}%
 \inddd{Y_{t_{1,n}} = y_0(t_{1,n})}}.
\]
We shall show that $M_n\sim M_{n,1}$ as $n\to\infty$. Indeed, we have that $Y_{t_{k,n}}\le y_0(t_{k,n})$ and hence $\psi(s_{k,n},Y_{t_{k,n}})\le \psi(s_{k,n},y_0(t_{k,n}))$ almost surely.
First, observe that
\[
\proba(Y_{t_{1,n}} = y_0(t_{1,n})) = \proba(Y_{t_{k,n}} = y_0(t_{k,n}), k=1,\dots, d+1).
\]
That is, once the process $Y_s$ reaches the highest point $y_0(s)$ at some time $s$, necessarily $s>1/A^2$, $Y_s$ stays on the deterministic trajectory $(y_0(t))_{t\ge s}$. This follows by computing $\proba(Y_t=y_0(t)\mid Y_s = y_0(s))$ for $1/A^2<s<t$. In this case, one has $y_0(s)>1$,
\[
y_0(s) +\sqrt{y_0(s)^2-1} = A\sqrt s, \qmand y_0(s) -\sqrt{y_0(s)^2-1} = \frac1{A\sqrt s}.
\]
So by~\eqref{Y-transitions} and~\eqref{eq:p0},
\begin{align*}
\proba(Y_t = y_0(t)\mid Y_s = y_0(s)) & = \nu\pp{\{y_0(t)\};A\sqrt t,B\sqrt t, \sqrt{s/t} A\sqrt s,\sqrt{s/t}/(A\sqrt s)}\\
& = p\pp{y_0(t);A\sqrt t,B\sqrt t,As/\sqrt t,1/(A\sqrt t),q}\\
& = \frac{\qp{1/(A^2t),ABs,B/A,s/t}}{\qp{B/A,s/t,1/(A^2t),ABs}} = 1.
\end{align*}
Introduce also %
$\mathfrak p_{0,n} = \proba(Y_{t_{1,n}} = y_0(t_{1,n}))$.
Recalling~\eqref{eq:p0}, we have
\begin{align*}
\mathfrak p_{0,n} &= p\pp{y_0(t_{1,n}); A\sqrt{t_{1,n}},B\sqrt{t_{1,n}},C/\sqrt{t_{1,n}},D/\sqrt{t_{1,n}},q}\\
& = \frac{\qp{1/(A^2t_{1,n}),BC,BD,CD/t_{1,n}}}{\qp{B/A,C/(At_{1,n}),D/(At_{1,n}),ABCD}}\\
& \to \frac{\qp{1/A^2,BC,BD,CD}}{\qp{B/A,C/A,D/A,ABCD}}= \mathfrak p_0
\end{align*}
as $n\to\infty$.
Therefore, by~\eqref{eq:psi},
\[
 M_{n,1}  = \mathfrak p_{0,n} {\prodd k1{d+1}\psi\pp{s_{k,n},y_0(t_{k,n})}^{n_k-n_{k-1}}}   \to  \mathfrak p_0\exp\pp{\summ k1d\frac{A}{2(1+A)^2}s_k^2 %
(x_k-x_{k-1})}
\]
as $n\to\infty$.
On the other hand, introducing $y_1^*(s) = \max\{y_1(s),1\}$, on the event $\{Y_s \ne y_0(s)\}$, we have $Y_s\le y_1^*(s)$ almost surely. Then,
\begin{align*}
M_{n,2} \le  (1-\mathfrak p_{0,n})\psi(s_{1,n},y_1^*(t_{1,n}))^{n_1}{\prodd k2{d+1}\psi\pp{s_{k,n},y_0(t_{k,n})}^{n_k-n_{k-1}}}.
\end{align*}
Now, one sees immediately that, since
by continuity
$\limn y_1^*(t_{1,n})/y_0(t_{1,n})=y_1^*(1)/y_0(1)\in(0,1)$,
\begin{align*}
\frac{M_{n,2}}{M_n} & \le  \frac{1-\mathfrak p_{0,n}}{\mathfrak p_{0,n}}\pp{\frac{\psi(s_{1,n},y_1^*(t_{1,n}))}{\psi(s_{1,n},y_0(t_{1,n}))}}^{n_1} \\
&
= \frac{1-\mathfrak p_{0,n}}{\mathfrak p_{0,n}}\pp{\frac{1+e^{-2s_{1,n}}+2e^{-s_{1,n}}y_1^*(t_{1,n})}{1+e^{-2s_{1,n}}+2e^{-s_{1,n}}y_0(t_{1,n})}}^{n_1} \to 0
\end{align*}
as $n\to\infty$.
Therefore $M_n\sim M_{n,1}$, and the desired result follows.
\end{proof}
\subsection{Proof
of Theorem \ref{T2}
for the low density phase $C>1, AC<1$
}\label{sec:particle_hole}
The result for the low density phase  is an immediate consequence of the result for the high density phase, by the particle-hole duality which we now explain.
We have seen the definition of an ASEP with parameters
$(\alpha,\beta,\gamma,\delta,q)$. Instead of thinking of particles jumping around, we view the particles as background and allow the holes to jump around (viewing
Figure~\ref{Fig1} as white particles jumping around among black unoccupied sites).
In this way, equivalently a hole jumps to the unoccupied left and right sites with rates 1 and $q$, respectively, and disappears at site $1$ with rate $\alpha$ and
at
site $n$ with rate $\delta$, and enters site $n$ if unoccupied with rate $\beta$ and site $1$ if unoccupied with rate
$\gamma$. This is the ASEP with parameters
$(\beta,\alpha,\delta,\gamma,q)$,
 if we relabel the sites $\{1,\dots,n\}$ by $\{n,\dots,1\}$.

Fix $q\in[0,1)$. Let
$\pi_n^{A,B,C,D}$
 denote the stationary distribution of the ASEP with parameters $(\alpha,\beta,\gamma,\delta,q)$.
Let $\tau_1,\dots,\tau_n$ be as before, and set $\varepsilon_j = 1-\tau_{n-j+1}$. Introduce
\[
\what h_n^{\rm L}(x) = \summ j1{\floor {nx}}\pp{\varepsilon_j-\frac C{1+C}}.
\]
The above argument shows that $\{\what h_n^{\rm L}(x)\}_{x\in[0,1]}$ with respect to
$\pi_n^{A,B,C,D}$ has the same law as $\{h^{\rm H}_n(x)\}_{x\in[0,1]}$ (defined in~\eqref{eq:h_HD}) with respect to
$\pi_n^{C,D,A,B}$.
Therefore, the high density phase of Theorem~\ref{T2} tells that
\[
\frac1{\sqrt n}\ccbb{{\what h_n^{\rm L}(x)}}_{x\in[0,1]}\fddto \frac{\sqrt C}{1+C}\ccbb{\B_x}_{x\in[0,1]}.
\]
We are interested in $h_n^{\rm L}(x) = \summ j1{\floor{nx}}(\tau_j-%
1/(1+C))$
with respect to $\pi_n^{A,B,C,D}$.
 Observe that
\[
\pp{\eps_j-\frac{C}{1+C}}+\pp{\tau_{n-j+1}-\frac{1}{1+C}}=0,
\] so
\equh\label{eq:particle_hole}
h_n^{\rm L}(x) + \pp{\what h_n^{\rm L}(1) - \what h_n^{\rm L}(1-x)} = \left\{
\begin{array}{ll}
0 & nx =\floor{nx}\\ \\
\displaystyle \frac C{1+C}- \varepsilon_{\floor{n(1-x)}+1} & nx\ne \floor{nx}
\end{array}
\right..
\eque
Since the error term is uniformly bounded,
the finite-dimensional distributions of %
$n^{-1/2}\sccbb{h^{\rm L}_n(x)}_{x\in[0,1]}$ have the same limit
as the finite-dimensional distributions of %
$n^{-1/2}\sccbb{\what h_n^{\rm L}(1-x) - \what h_n^{\rm L}(1)}_{x\in[0,1]}$, and
 we arrive at
\[
\frac1{\sqrt n}\ccbb{h^{\rm L}_n(x)}_{x\in[0,1]} \fddto \frac{\sqrt C}{1+C}\ccbb{\B_{1-x} - \B_{1}}_{x\in[0,1]} \eqd \frac{\sqrt C}{1+C}\ccbb{\B_x}_{x\in[0,1]}.
\]
This proves the low density phase of Theorem~\ref{T2}.
\section{
Proofs
for Maximal current phase and its boundary}\label{sec:T1}
The proofs for  the two cases are very similar and are hence unified.
We need some preparation for the proof.
In Section~\ref{sec:Z} we review an important auxiliary Markov process $\Z$, and in particular how this Markov process shows up in the Laplace representations of Brownian excursion and meander. Another important role of this Markov process is that it is the tangent process of the Askey--Wilson process at the boundary. This result, playing a central role in the proof, will be established first  in Proposition~\ref{prop:tangent} in Section~\ref{sec:AW_continuous}. The
case $A\le 1, C<1$ is then
proved
 in Theorem~\ref{T1''} in Section~\ref{sec:A<=1C<1}. The case $A<1, C=1$ is proved by the particle-hole duality in Section~\ref{sec:A<1C=1}.
\subsection{An auxiliary Markov process}\label{sec:Z}
An auxiliary Markov process, denoted by $\Z$ in the rest of the paper, will play a crucial role in the proof for the maximal current phase and its boundary. This is a positive self-similar Markov process with values in $[0,\infty)$ and transition probability density function
\equh\label{eq:Z_pdf}
\qqq_{s,t}(x,y) = \frac{2(t-s)\sqrt y}{\pi\bb{(t-s)^4+2(t-s)^2(x+y) + (x-y)^2}} \indd{x\ge 0, y\ge 0}, \quad s<t.
\eque
This process is self-similar in the sense that, letting $\proba_x$ denote the law of $\Z$ starting at $\Z_0 = x$,
\[
\pp{\ccbb{\Z_{\lambda t}}_{t\ge 0},\proba_x} \eqd \pp{\lambda^2\ccbb{\Z_t}_{t\ge 0},\proba_{x/\lambda^2}} \mfa \lambda,x>0.
\]
This process has not been much investigated in the literature,
except for %
a series of recent
papers \citep{bryc16local,wang18extremes,bryc18dual}.
In \citep{bryc16local}, when investigating the path properties of so-called $q$-Gaussian processes, we proved that the process $\Z$ arises as their tangent process at the boundary.
(We also proved in Section 3 therein that the transformed process $\wt \Z$ via $\wt \Z_t=\Z_{t/2}+t^2/4$ has already shown up in the literature: in a general framework connecting non-commutative stochastic process and classical Markov process developed by \citet{biane98processes}. In this framework, $\wt \Z$ as a classical Markov process corresponds to the free $1/2$-stable process, the knowledge of which we do not need here.)

Recently, we also found out in  \citep{bryc18dual} that the process $\Z$
plays an intriguing role in representations of the Laplace transforms %
 of finite-dimensional distributions of Brownian excursion and Brownian meander.
Write $\esp_u(\cdot) = \esp(\cdot\mid \Z_0 = u)$.
For all $d\in\N$,  $s_1>s_2>\cdots>s_d>s_{d+1} = 0$, and $0 = x_0 < x_1<\cdots<
x_d \leq
  x_{d+1} = 1$, we have shown in \citep{bryc18dual} that
\begin{multline}\label{eq:ex}
 \esp\exp\pp{-\summ k1d(s_k-s_{k+1})\B_{x_k}^{ex}}
\\
= \frac1{\sqrt{2\pi}}\int_{\R_+}u^{1/2} du\esp _u \exp\pp{-\frac12\summ k1{d+1}\Z_{s_k}(x_k-x_{k-1})},
\end{multline}
and
\begin{multline}\label{eq:me}
 \esp\exp\pp{-\summ k1d(s_k-s_{k+1})\B_{x_k}^{me}}
\\
= \frac1{\sqrt{2\pi}}\int_{\R_+}\frac1{u^{1/2}} du\esp _u \exp\pp{-\frac12\summ k1{d+1}\Z_{s_k}(x_k-x_{k-1})}.
\end{multline}
Note that the formulae here are obtained via the changes of variables
\[
s_k = \what s_{d+1-k} \qmand  x_k = 1-t_{d+1-k}, \quad k=1,\dots,d+1,
\]
and $x_0 = 0$, where $\what s_k$ and $t_k$ correspond to variables $s_k,t_k$ used in \citep{bryc18dual}.

The above identities
can be
 obtained by direct computation using the
joint density functions
  of Brownian excursion and meander.
These  explicit densities  of the two processes will not be used
 in this paper.
Standard references %
about Brownian excursions and meanders
 include \citep{revuz99continuous,pitman99brownian, pitman06combinatorial}.

\subsection{Tangent process of Askey--Wilson process with $A\le1$ and $C<1$}\label{sec:AW_continuous}
The proof is essentially based on the laws of the Askey--Wilson process %
 $\{Y_t\}_{t\in[1-\varepsilon,1]}$ for some $\varepsilon>0$ small enough. With $A\le 1$ and $C<1$,
for this range of $t$
the marginal and transition probability laws are absolutely continuous with respect to the Lebesgue measure, with compact support on $[-1,1]$.
For the purpose of computing asymptotics, we express the formula by regrouping factors into those that  tend to a
non-zero
constant as $t\uparrow 1$ and $x\uparrow 1$, and those that tend to
zero. (Some factors go to
zero
only when $A=1$, and we include them in the second group.) In particular,
the Askey--Wilson process $Y$ has the marginal probability density function
\begin{align}
  \label{f-density}
  \pi_t(x) & = f\pp{x;A\sqrt t,B\sqrt t,C/\sqrt{t},D/\sqrt t,q} \\
&  = \frac{\qp q\qp{ABt,AC,AD,BC,BD,CD/t}}{2\pi\qp{ABCD}|\qp{B\sqrt te^{i\theta_x},C e^{i\theta_x}/\sqrt{t},De^{i\theta_x}/\sqrt t}|^2}\nonumber\\
& \quad\times \frac{\sabs{\qp{e^{2i\theta_{x}}}}^2}{\sqrt{1-x^2}|\qp{A\sqrt te^{i\theta_x}}|^2},%
\nonumber
\end{align}
with $x=\cos\theta_x$, for $x\in[-1,1]$,
 and
transition probability density function
\begin{align}\label{eq:pst}
p_{s,t}(x,y) & = f\pp{y;A\sqrt t,B\sqrt t,\sqrt{s/t}e^{i\theta_x},\sqrt{s/t}e^{-i\theta_x}} \\
& = \frac{\qp q\qp{ABt}\sabs{\qp {B\sqrt se^{i\theta_{x}}}}^2}{2\pi\qp{ABs}\sabs{\qp {B\sqrt te^{i\theta_{y}}}}^2}\nonumber\\
& \quad\times
\frac{\sabs{\qp {A\sqrt se^{i\theta_x},e^{2i\theta_{y}}}}^2\qp {s/t}}{\sqrt{1-y^2}\sabs{\qp{A\sqrt te^{i\theta_y},\sqrt{s/t}e^{i(\theta_{x}+\theta_{y})},\sqrt{s/t}e^{i(-\theta_x+\theta_y)}}}^2},\nonumber
\end{align}
where $y=\cos\theta_y$, for $x,y\in[-1,1]$. (Recall %
\eqref{eq:f} and
$q$-Pochhammer notation in \eqref{eq:Pochhamer}.)

It turns out to be helpful  to
consider
$\{Y_t\}_{t\in[1-\varepsilon,1]}$ in the reversed time direction
when $Y_1$ is close to 1.
 For this purpose, we introduce
 Markov process
\equh\label{eq:Yhat} \what Y_{s}\topp n=2n\pp{1-Y_{e^{-2s/\sqrt n}}}, s\ge 0. \eque
(The parameterization for $\what Y_s\topp n$ is chosen  so that we have the corresponding convergence to the tangent process \eqref{eq:conv_tangent} below.) %
The following limit theorem of $\what Y\topp n$ is at the core of the proof of Theorems~\ref{T1} and~\ref{T1'}. Let $\what \pi\topp n_{s}$ denote the probability density function of $\what Y\topp n_{s}$ %
and $\what p_{s,t}\topp n$ the transition probability density of $\what Y\topp n$ from time $s$ to $t$.

\begin{proposition}\label{prop:tangent}Under the notation above, for $A\le 1$ and $C<1$,
$0\leq s<t$ and $u,v\geq 0$ we have
$$\limn\what p_{s,t}\topp n(u,v)= \qqq_{s,t}(u,v).$$
In particular, we have
\equh\label{eq:conv_tangent}
\mathcal{L}\pp{\ccbb{\what Y_{s}\topp n}_{s\ge 0}\mmid \what Y_0\topp n = u}
\fddto \mathcal{L}\pp{\ccbb{\Z_{s}}_{s\ge 0}\mmid \Z_0 = u}, \mfa u>0,
\eque
where the left-hand side is understood as the law of $\what Y$ given $\what Y_0\topp n = u$, and similarly for the right-hand
side above. Moreover, %
as $n\to\infty$
\begin{equation}
  \label{eq:what-pi}
  \what \pi\topp n_0(u)
\sim \left\{
\begin{array}{ll}
\displaystyle \frac{\mathfrak c_1}{n^{3/2}}\cdot u^{1/2} & A<1,C<1\\
\\
\displaystyle \frac{\mathfrak c_2}{n^{1/2}}\cdot\frac 1{u^{1/2}} & A=1, C<1,
\end{array}
\right.
 \mfa u>0,
\end{equation}
\comment{
for all $u>0$, we have
\begin{equation}
  \label{Lim_hat_pi1}
  \limn n^{3/2}\what \pi\topp n_0(u) = \mathfrak c_1\cdot u^{1/2} \mbox{ if $A<1, C<1$},
\end{equation}
and
\begin{equation} \label{Lim_hat_pi2}
\limn n^{1/2}\what \pi\topp n_0(u) = \mathfrak c_2\cdot u^{-1/2} \mbox{ if $A=1, C<1$},
\end{equation}
}
with
\begin{align*}
\mathfrak c_1& = \frac{\qp q^3}\pi \frac{\qp{AB,AC,AD,BC,BD,CD}}{\qp{ABCD}\qp{A,B,C,D}^2}\\
\mathfrak c_2 & = \frac{\qp q}\pi\frac{\qp{BC,BD,CD}}{\qp{BCD,B,C,D}} %
,
\end{align*}
and  there exists %
a
constant $c$ such that for all
$n$ large enough,
\equh\label{eq:pihat_bd}
\what \pi\topp n_0(u) \le \begin{cases}
\displaystyle \frac c{n^{3/2}}\cdot u^{1/2} & A<1, C<1\\ \\
\displaystyle \frac c{n^{1/2}}\cdot \frac 1{u^{1/2}} & A=1, C<1,
\end{cases}
\mfa u>0.
\eque
\end{proposition}

\begin{remark}
Recall that notation $a_n\sim b_n$ means that   \eqref{eq:what-pi} reads as
\[
\limn n^{3/2}\what \pi\topp n_0(u) = \mathfrak c_1\cdot u^{1/2}
\]
if $A<1,C<1$, and
\[\limn n^{1/2}\what \pi\topp n_0(u) = \mathfrak c_2\cdot u^{-1/2}\]
if $A=1, C<1$.
\end{remark}

The
first part of the proposition
implies
that the %
(time-reversed) tangent process of $Y$ at the upper boundary
of the support of $Y_1$
 is
 process
 $\Z$. The role of the second part will become clear soon.
It is remarkable that  for different choices of $A$ and $C$,  the tangent processes are the same, but the initial laws ($\what \pi_0\topp n$) in the limit are different and of different normalization %
orders.
Similar results on tangent processes have been known for closely related processes \citep{bryc16local,wang18extremes}.
 We expect that the finite-dimensional convergence can be strengthened to weak convergence in $D([0,1])$ by a similar treatment as in \citep{wang18extremes},
but this is not needed in this paper. %
We first compute some   asymptotics
of
the   Askey--Wilson process $Y$.
\begin{lemma}%
For $u,v,s,t>0, s<t$ and
\[
x_n = 1-\frac {u}{2n}, y_n = 1-\frac {v}{2n}, s_n = e^{-2s/\sqrt n}, t_n = e^{-2t/\sqrt n},
\]%
\equh\label{eq:marginal_convergence}
\pi_{s_n}(x_n) \sim \left\{
\begin{array}{ll}
\displaystyle \frac{2\mathfrak c_1}{\sqrt n}\cdot \sqrt{u} & A<1,C<1\\
\\
\displaystyle 2\mathfrak c_2\sqrt n\cdot \frac {\sqrt u}{s^2+u} & A=1, C<1,
\end{array}
\right.
\eque
and there exists some constant $c$ such that
\equh\label{eq:pi_bd}
\pi_{s_n}(x_n) \le \begin{cases}
\displaystyle c\sqrt{\frac un} & A<1, C<1\\
\\
\displaystyle c\sqrt{\frac nu} & A=1,C<1,
\end{cases}
\eque
for all $n$ large enough.
Moreover, %
\equh\label{eq:Markov_convergence}
 p_{t_n,s_n}(y_n,x_n) \sim
\left\{
\begin{array}{ll}
\displaystyle 2n\cdot \qqq_{s,t}(v,u) & A<1, C<1\\ \\
\displaystyle 2n\cdot \qqq_{s,t}(v,u) \frac{t^2+v}{s^2+u}& A=1, C<1.
\end{array}
\right.
\eque
\end{lemma}

\begin{proof}
We first establish
\equh\label{eq:1}
\abs{\qpp {\sqrt{{t_n}/{s_n}}e^{\pm i\theta_{x_n}}e^{i\theta_{y_n}}}}^2  \sim \frac1n\bb{(t-s)^2+(\sqrt u\pm \sqrt v)^2}\qp q^2.
\eque
For this, write
\begin{multline}\label{eq:2}
\abs{\qpp {\sqrt{{t_n}/{s_n}}e^{\pm i\theta_{x_n}}e^{i\theta_{y_n}}}}^2   \\=
\abs{1-\sqrt{t_n/s_n}e^{i(\pm \theta_{x_n}+\theta_{y_n})}}^2 \abs{\qpp{\sqrt{t_n/s_n}e^{i(\pm\theta_{x_n}+\theta_{y_n})}q}}^2.
\end{multline}
Since $t_n\to 1, s_n\to 1, \theta_{x_n}\to 0, \theta_{y_n}\to 0$,
the second factor above is asymptotically equivalent to $\qp q^2$. For the first factor,
\begin{align}
\abs{1-\sqrt{t_n/s_n}e^{i(\pm \theta_{x_n}+\theta_{y_n})}}^2  & = 1+t_n/s_n - 2\sqrt{t_n/s_n}\pp{x_ny_n \mp  \sqrt{1-x_n^2}\sqrt{1-y_n^2}}\label{eq:3}\\
& \sim \frac1n\bb{(t-s)^2+(\sqrt u\pm \sqrt v)^2}.%
\nonumber
\end{align}
This proves~\eqref{eq:1}. As special cases, we have %
\begin{align}
\label{eq:s^2+u}\abs{\qpp{\sqrt{s_n}e^{i\theta_{x_n}}}}^2& \sim \frac1n\pp{{s^2}+u}\qp q^2,\\
\nonumber\qpp{{t_n}/{s_n}}& \sim \frac{2(t-s)}{\sqrt n}\qp q,\\
\nonumber\abs{\qpp{e^{i\theta_{x_n}}}}^2 & \sim \frac{u}n\qp q^2,\\
\nonumber\abs{\qpp{e^{2i\theta_{x_n}}}}^2 & \sim \frac {4u}n\qp q^2,
\end{align}
as $n\to\infty$.

We now examine $\pi_{s_n}(x_n)$ using~\eqref{f-density}.
The pointwise asymptotics~\eqref{eq:marginal_convergence} are straightforward to obtain. We prove the upper bounds.
Recall that $\pi_{s_n}$ has support on $[-1,1]$. Therefore from now on we assume $x_n \in[-1,1]$, or equivalently $u\in[0,4n]$.
We first focus on the second fraction of $\pi_{s_n}(x_n)$ in the expression~\eqref{f-density},
\[
\wt \pi_{s_n}(x_n) = \frac{\abs{\qp {e^{2i\theta_{x_n}}}}^2}{\sqrt{1-x_n^2}\abs{\qp{A\sqrt {s_n}e^{i\theta_{x_n}}}}^2}.
\]%
Let $c$ denote a constant independent of $n$, but may change from line to line.
Suppose $A<1$ first.  Observe that
$\abs{\qp{A\sqrt {s_n}e^{i\theta_{x_n}}}}^2\ge\qp A^2$.
Furthermore, we have $|\qp{e^{2i\theta_{x_n}}}|^2 = 4(1-x_n^2)|\qp{e^{2i\theta_{x_n}}q}|^2$ (see e.g.~\eqref{eq:2} and~\eqref{eq:3}). Therefore,
\[
\wt \pi_{s_n}(x_n)\le c\frac{\abs{\qp{e^{2i\theta_{x_n}}}}^2}{\sqrt{1-x_n^2}} \le c\sqrt{1-x_n^2} \le c\sqrt{\frac un}.
\]
Now suppose $A=1$.
Then
\[
\wt \pi_{s_n}(x_n)\le c\frac{\abs{\qp{e^{2i\theta_{x_n}}}}^2}{\sqrt{1-x_n^2}\abs{\qp{\sqrt {s_n}e^{i\theta_{x_n}}}}^2}\le c\frac{\sqrt{1-x_n^2}}{\abs{1-\sqrt {s_n}e^{i\theta_{x_n}}}^2}.
\]
By~\eqref{eq:3}, $\abs{1-\sqrt{s_n}e^{i\theta_{x_n}}}^2 = 1+s_n-2\sqrt{s_n}x_n =
(1-\sqrt{s_n})^2
+2\sqrt{s_n}(1-x_n)
>2 \sqrt{s_1}(1-x_n)
$. So,
we see that $\wt \pi_{s_n}(x_n)$
is bounded from above by,
\[
c\frac{\sqrt{1-x_n^2}}{1-x_n} \le c\sqrt{\frac nu}.
\]
So we see that  $\wt \pi_{s_n}(x_n)$ can be controlled by
the
bounds in ~\eqref{eq:pi_bd}. For the first fraction of $\pi_{s_n}(x_n)$ in~\eqref{f-density}, it suffices to control
\[
\frac{\qp{CD/s_n}}{\qp{Ce^{i\theta_{x_n}}/\sqrt{s_n},De^{i\theta_{x_n}}/\sqrt{s_n}}}.
\]
For the numerator,
since $CD\in(-1,0]$,  we have $\qp{CD/s_n} \le \qp{CD/s_1}$.
For the denominator,
 for $n$ large enough so that $C/\sqrt {s_n} <(1+C)/2<1$, we have
  $\abs{\qp{Ce^{i\theta_{x_n}}/\sqrt{s_n}}}^2\ge \qp{(1+C)/2}^2$.
 Similarly, $|\qp{De^{i\theta_{x_n}}/\sqrt{s_n}}|^2>\qp{(1-D)/2}$   for $n$ large enough so that $-D/\sqrt {s_n} <(1-D)/2<1$.
 So the first fraction can be bounded by some
 constant $c$. This completes the proof of~\eqref{eq:pi_bd}.

 Now to show~\eqref{eq:Markov_convergence},
observe first that
by  \eqref{eq:1}  applied to each factor, %
\begin{align*}
& \abs{\qpp{\sqrt{t_n/s_n}e^{i(\theta_{x_n}+\theta_{y_n})}}\qpp{\sqrt{t_n/s_n}e^{i(-\theta_{x_n}+\theta_{y_n})}}}^2\\
& \sim \frac1{n^2}\qp q^4 \bb{(t-s)^2+(\sqrt u+ \sqrt v)^2}\bb{(t-s)^2+(\sqrt u- \sqrt v)^2}\\
& = \frac1{n^2}\qp q^4\bb{(t-s)^4+2(t-s)^2(u+v) + (u-v)^2}.
\end{align*}
So~\eqref{eq:pst} now yields
\begin{align*}
p_{t_n,s_n}(y_n,x_n) & \sim \frac{\qp q}{2\pi}\frac{\abs{\qp{A\sqrt{t_n}e^{i\theta_{y_n}}}}^2}{\abs{\qp{A\sqrt{s_n}e^{i\theta_{x_n}}}}^2}\\
& \quad \times \frac{\abs{\qp{e^{2i\theta_{x_n}}}}^2\qp{t_n/s_n}}{\sqrt{1-x_n^2}\abs{\qp{\sqrt{t_n/s_n}e^{i(\theta_{y_n}+\theta_{x_n})}}\qp{\sqrt{t_n/s_n}e^{i(-\theta_{y_n}+\theta_{x_n})}}}^2}\\
& \sim \frac{\qp q}{2\pi}\frac{\abs{\qp{A\sqrt{t_n}e^{i\theta_{y_n}}}}^2}{\abs{\qp{A\sqrt{s_n}e^{i\theta_{x_n}}}}^2}\\
& \quad \times \frac{\displaystyle\frac{4u}n\qp q^2\frac{2(t-s)}{\sqrt n} \qp q}{\displaystyle\sqrt{\frac un}  \frac1{n^2}\qp q^4\bb{(t-s)^4+2(t-s)^2(u+v) + (u-v)^2} }\\
& = \frac{\abs{\qp{A\sqrt{t_n}e^{i\theta_{y_n}}}}^2}{\abs{\qp{A\sqrt{s_n}e^{i\theta_{x_n}}}}^2} \cdot 2n \cdot \qqq_{s,t}(v,u).
\end{align*}
By~\eqref{eq:s^2+u}, we have
\[
\limn\frac{\abs{\qp{A\sqrt{t_n}e^{i\theta_{y_n}}}}^2}{\abs{\qp{A\sqrt{s_n}e^{i\theta_{x_n}}}}^2}
= \begin{cases}
1 & A<1\\
\\
\displaystyle \frac{t^2+v}{s^2+u}& A=1.
\end{cases}
\]
This completes the proof.\end{proof}

\begin{proof}[Proof of Proposition~\ref{prop:tangent}]
Recall the relation between $\what Y$ and $Y$ in~\eqref{eq:Yhat}. We have that
$\what \pi_s\topp n(u) = \pi_{s_n}(x_n)/(2n)$,
and so the second part of the proposition follows from~\eqref{eq:marginal_convergence} and~\eqref{eq:pi_bd}.
 Fix $0<s<t$ now, so that $0<t_n<s_n<1$. For the transition probability density of $\what Y$, we have

\[
\what p_{s,t}\topp n(u,v) = p_{s_n,t_n}(x_n,y_n)\frac1{2n} = p_{t_n,s_n}(y_n,x_n)\frac{\pi_{t_n}(y_n)}{\pi_{s_n}(x_n)}\frac1{2n},
\]
where $p_{s_n,t_n}$ denotes the transition probability of $Y$ in the reversed time direction. In the case $A<1,C<1$,
from~\eqref{eq:marginal_convergence} and \eqref{eq:Markov_convergence} we get
\[
\what p_{s,t}\topp n(u,v) \sim \qqq_{s,t} (v,u)\sqrt{\frac vu} =
\qqq_{s,t}(u,v)
\]
directly.
In the case $A=1,C<1$,
from~\eqref{eq:marginal_convergence} and  \eqref{eq:Markov_convergence} we get
\[
\what p_{s,t}\topp n(u,v) \sim
\qqq_{s,t}(v,u)
\frac{t^2+v}{s^2+u}\cdot\frac{\sqrt v/(t^2+v)}{\sqrt u/(s^2+u)} =
\qqq_{s,t}(u,v).
\]
Since the transition densities determine conditional finite-dimensional densities, the finite-dimensional (conditional) densities converge. Therefore, by Scheff\'e's %
Theorem
the finite-dimensional (conditional) distributions converge weakly for every $u>0$, which
completes the proof for the first part of the proposition.
\end{proof}

\subsection{Proof
of Theorems \ref{T1} and \ref{T1'}
for the case $A\le 1, C<1$}\label{sec:A<=1C<1}
For $d\in\N$, $c_1,\dots,c_d>0$,
$x_0=0<x_1<\cdots<x_d\le x_{d+1}=1$, introduce the Laplace transform
\begin{align*}
\varphi_{\vv x,n}(\vv c) & = \aa{\exp\pp{-\summ k1d c_kh_n(x_k)}}_n
\\
& = \aa{\exp\pp{-\summ k1d\sum_{j=\floor{nx_{k-1}}+1}^{\floor{nx_k}}\pp{\tau_j-\frac12}(c_k+\cdots+c_d)}}_n.
\end{align*}
This time it will be more convenient to write
 \equh\label{eq:s_k}
 s_k%
 =
  \frac12(c_k+\cdots+c_d),  k=1,\dots,d
 \eque
(notice the extra $1/2$ compared to $s_k$ in previous sections),
$s_{d+1} = 0$, and $n_k =\sfloor{n{x_k}}, k=1,\dots, d+1$.
We have, by Theorem \ref{T-BW17} again,
\begin{align*}
\varphi_{\vv x,n}(\vv c) %
& =  \exp\pp{\summ k1d s_k(n_k-n_{k-1})}\aa{{\prodd k1{d+1}\prod_{j=n_{k-1}+1}^{n_k}(e^{-2s_k})^{\tau_j}}}_n\\
& = \exp\pp{\summ k1d s_k(n_k-n_{k-1})}\frac{\esp\bb{\prodd k1{d+1}(1+e^{-2s_k}+2e^{-s_k}Y_{e^{-2s_k}})^{n_k-n_{k-1}}}}{2^n\esp(1+Y_1)^n}\\
& = \frac{\esp\bb{\prodd k1{d+1} \pp{\cosh(s_k)+Y_{e^{-2s_k}}}^{n_k-n_{k-1}}}}{Z_n},
\end{align*}
where we write
$Z_n = \esp(1+Y_1)^n$.
In order to establish the convergence of finite-dimensional
distributions
of $h_n$, we compute the limit of
\equh\label{eq:Laplace}
\varphi_{\vv x,n}\pp{\frac{\vv c}{\sqrt n}} = \frac{\esp\bb{\prodd k1{d+1} \pp{\cosh(s_k/\sqrt n)+Y_{e^{-2s_k/\sqrt n}}}^{n_k-n_{k-1}}}}{Z_n},
\eque
as $n\to\infty$, and identify the limit with the Laplace transform of the corresponding process.
Now
in view of Theorem \ref{Thm-Laplace},
Theorems~\ref{T1} and~\ref{T1'} are a consequence of the following.
\begin{theorem}\label{T1''}
For all $d\in\N, 0<x_1<\cdots<x_d\le 1$ and $\vv c = (c_1,\dots,c_d)\in(0,\infty)^d$,
\begin{multline}\label{eq:Laplace_convergence}
\limn\varphi_{\vv x,n}\pp{\frac{\vv c}{\sqrt n}}
\\
=\left\{
\begin{array}{ll}
\displaystyle \esp\exp\pp{-\summ k1d \frac{c_k}{2\sqrt 2}\B_{x_k}}\esp \exp\pp{-\summ k1d\frac{c_k}{2\sqrt 2}\B^{ex}_{x_k}} & A<1, C<1,\\ \\
\displaystyle  \esp\exp\pp{-\summ k1d \frac{c_k}{2\sqrt 2}\B_{x_k}}\esp \exp\pp{-\summ k1d\frac{c_k}{2\sqrt 2}\B^{me}_{x_k}} & A=1, C<1.
 \end{array}
 \right.
\end{multline}

\end{theorem}
We first
determine
the asymptotics of $\esp(1+Y_1)^n$ in~\eqref{eq:Laplace}.

\begin{lemma}\label{lem:limit_Zn} We have
\[
Z_n = \esp(1+Y_1)^n
 \sim\left\{
\begin{array}{ll}
\displaystyle\frac{2^n}{n^{3/2}}\cdot 4\sqrt \pi\cdot \mathfrak c_1
& A< 1, C < 1\\
\\
\displaystyle
 \frac{2^n}{n^{1/2}}\cdot 2\sqrt \pi\cdot \mathfrak c_2
& A=1, C<1.%
\end{array}
\right.
\]
where $\mathfrak c_1, \mathfrak c_2$ are defined in Proposition~\ref{prop:tangent}.
\end{lemma}
\begin{remark}\label{rem:Z_n}
The quantity $Z_n=\esp(1+Y_1)^n$ is closely related to  the {\em partition function} in the literature, denoted by $\mathcal Z_n$ %
for the discussion here, via (see explanation in \cite[Remark 5]{bryc17asymmetric}),
\[
\mathcal Z_n = Z_n\frac{2^n\<W|V\>}{(1-q)^{n}},
\]
where the factor $\<W|V\>$
  depends on the  choice of the vectors in \eqref{MatrixAnsatz} and   needs to be taken into account
when comparing Lemma \ref{lem:limit_Zn} with the literature.
Partial results on asymptotics of partition function, including
also low/high density phases, have been known. See for example \cite[(52),(53) and (55)]{derrida93exact}
for the case %
$q=0,\gamma=\delta = 0$ (with $\<W|V\> = 1$)  and
\cite[(56)]{blythe00exact} for
$A<1,C<1, B=D=0$ with $\<W|V\>= 1/\qp{AC}$.
In more generality,
\citet[(6.6) and (6.9)]{uchiyama04asymmetric}
compute $\mathcal Z_n$
for   $A>1, A>C$ and  $A,C<1$
(with $\<W|V\>={\qp{ABCD}}/{\qp{q, AB,AC,AD,BC,BD,CD}}$).
We do not find
 general results on asymptotic of $Z_n$
 for  $A=1,C<1$ in the literature.

\end{remark}

\begin{proof}
It follows from~\eqref{eq:marginal_convergence} that, taking $s=0$ therein,
as $n\to\infty$, for $u>0$,
\begin{equation}\label{aadvark1}
\pi_1\pp{1-\frac u{2n}}\sim\left\{
\begin{array}{ll}
\displaystyle 2\mathfrak c_1\sqrt{\frac{u}n} & A< 1, C< 1\\
\\
\displaystyle 2\mathfrak c_2\sqrt{\frac nu} & A=1, C< 1.%
 \end{array}
 \right.
\end{equation}
Consider first the case $A=1$. Then
\begin{multline}
\label{eq:Zn2}
2^{-n}n^{1/2}\esp(1+Y_1)^n =2^{-n}n^{1/2}\int_{-1}^1(1+y)^n\pi_1(y)dy
\\
= \int_{\R_+} \indd{u \leq 4n}\pp{1-\frac u{4n}}^n \pi_1\pp{1-\frac u{2n}}\frac{du}{2n^{1/2}}.
\end{multline}
In view of   \eqref{aadvark1}, the integrand  on the right-hand side of \eqref{eq:Zn2} converges to  $\mathfrak c_2e^{-u/4}/\sqrt{u}$.
Since  $\int_0^\infty e^{-u/4}/\sqrt{u}du=2\sqrt{\pi}$, to  conclude the proof  by the dominated convergence theorem, we now give %
an
integrable bound for the integrands.
Recalling~\eqref{eq:pi_bd}, and noting that $\pi_1\pp{1-%
u/(2n)}=0$ for $u>2n$,
we have, for $n$ large enough, %
\begin{equation*}
  \pp{1-\frac u{4n}}^n\pi_1\pp{1-\frac u{2n}}
   \le ce^{-u/4}\sqrt{\frac nu}, \mfa u>0
\end{equation*}
for some constant $c$ that depends on $C$ and $q<1$.

 The proof for $A<1,C<1$ is similar, starting from~\eqref{eq:Zn2} and eventually by the dominated convergence theorem. In this process, the above bound  is replaced by, for $n$ large enough,
\begin{equation*}
\pp{1-\frac u{4n}}^n\pi_1\pp{1-\frac u{2n}}
 \le ce^{-u/4}\sqrt \frac un, \mfa u>0.
\end{equation*}
Details are omitted. %
\end{proof}
Next, we look at the numerator in~\eqref{eq:Laplace}. We write,
recalling
$\what Y$ in \eqref{eq:Yhat}, %
\begin{equation}
  \label{eq:cosh2G}
  2^{-n}\esp  \bb{\prodd k1{d+1} \pp{\cosh\pp{\frac{s_k}{\sqrt n}}+Y_{e^{-2s_k/\sqrt n}}}^{n_k-n_{k-1}}}
 = \esp G_n\pp{\what Y_{s_1}\topp n,\dots,\what Y_{s_{d+1}}\topp n}
\end{equation}
with
\begin{align*}
G_n(\vv u)& = 2^{-n}\prodd k1{d+1} \pp{\cosh\pp{\frac{s_k}{\sqrt n}}+1-\frac{u_k}{2n}}^{n_k-n_{k-1}} \indd{\vv u\le4n}
\\
& = \prodd k1{d+1}\pp{1+\sinh^2\pp{\frac{s_k}{2\sqrt n}} - \frac {u_k}{4n}}^{n_k-n_{k-1}}\indd{\vv u \le 4n}.
\end{align*}
Recall that $\what Y$ takes values from $[0,4n]$.
We introduce the indicator function above for the convenience in later analysis
($\vv u\le 4n$ stands for $\max_{k=1,\dots,d+1}u_k\le 4n$ here).
Since $(n_k-n_{k-1})/n\to x_k-x_{k-1}$, we have %
\begin{equation}
  \label{eq:LimG_n}
  \limn G_n(\vv u) = \exp\pp{\frac14\summ k1{d+1}(s_k^2-u_k)(x_k-x_{k-1})}=:G(\vv u).
\end{equation}
The key step  is to show the following. Recall that $s_1>\cdots>s_d>s_{d+1} = 0$.
\begin{proposition}\label{prop:limit_Gn} %
With the notations above and %
$\esp_u(\cdot) = \esp(\cdot\mid \Z_{0} = u)$,
\begin{multline*}
\esp G_n\pp{\what Y_{s_1}\topp n,\dots,\what Y_{s_{d}}\topp n,\what Y_{0}\topp n} \\
\sim\left\{
\begin{array}{ll}
\displaystyle\frac{\mathfrak c_1}{n^{3/2}}\int_{\R_+^{}}\esp_u G(\Z_{s_1},\dots,\Z_{s_{d}},u)u^{1/2}du & A<1, C<1\\ \\
\displaystyle\frac{\mathfrak c_2}{n^{1/2}}\int_{\R_+^{}}\esp_u G(\Z_{s_1},\dots,\Z_{s_{d}},u)\frac1{u^{1/2}}du & A=1, C<1.
\end{array}
\right.
\end{multline*}
\end{proposition}
\begin{proof}
We start with some properties of $G_n$ as preparations.
First, since $(1+x)^m\leq \exp(m x)$ for $m\in\NN, x\geq -1$, we get an exponential bound on $G_n$:
\begin{align}
G_n(\vv u) & \le \prodd k1{d+1}\exp\pp{\frac{n_k-n_{k-1}}{4n}\bb{4n\sinh^2\pp{\frac{s_k}{2\sqrt n}} - u_k}}\indd{\vv u\le 4n}\nonumber\\
& \le c \prodd k1{d+1}\exp\pp{-\frac{n_k-n_{k-1}}{4n}u_k}.\label{eq:Gn_bd}
\end{align}
Here and below, $c$ denotes a constant that does not depend on $\vv u$ and $n$, but may vary from line to line.
This inequality also shows that $G_n(\vv u)$ is uniformly bounded %
for all $\vv u\in\R_+^{d+1}$ and $n\in\N$.

Next, by the inequality $|\prodd k1n a_k - \prodd k1n b_k|\le  M^n \summ k1n |a_k-b_k|$ provided that $|a_k|,|b_k|\le M$, we have, %
for all $\vv u,\vv u'\in[0,4n]$,
\begin{align*}
\abs{G_n(\vv u)- G_n(\vv u')} & \le \pp{1+\sinh^2\pp{\frac {s_1}{2\sqrt n}}}^n\summ k1{d+1}(n_k-n_{k+1})\frac{|u_k-u_k'|}{4n} \\
& \le c\summ k1{d+1}\abs{u_k-u_k'}.
\end{align*}
It follows that for all $\vv u_n, \vv u\in\R_+^{d+1}$,
\begin{equation}\label{eq:Billingsley}
\limn G_n(\vv u_n) = G(\vv u), \mbox{ if } \vv u _n\to \vv u \in \R^{d+1}_+ \mmas n\to\infty.
\end{equation}

In the remaining part of the proof  with a little abuse of notation we write  $\esp_u(\cdot)$  which is to be understood as %
 $\esp(\cdot \mid \what Y_0\topp n = u)$ when dealing with  $\what Y\topp n$ and as $\esp(\cdot\mid \Z_{0} = u)$ when dealing with $\Z$.
 Conditioning on the value of $\what Y\topp n_{s_{d+1}}=\what Y_0\topp n$,  we write
\begin{equation}\label{aadvark2}
  \esp G_n\pp{\what Y_{s_1}\topp n,\dots,\what Y_{s_{d+1}}\topp n}  =
\int_0^\infty \esp_{u}G_n\pp{\what Y_{s_1}\topp n,\dots,\what Y\topp n_{s_{d}},u} \what \pi\topp n_0(u)du.
\end{equation}
Because of
the uniform boundedness, %
\eqref{eq:Billingsley} and the weak convergence of the tangent process established in Proposition~\ref{prop:tangent}, we have
\equh\label{eq:weak_convergence}
\limn\esp_{u}G_n\pp{\what Y_{s_1}\topp n,\dots, \what Y_{s_{d}}\topp n,u}=\esp_u G(\Z_{s_1},\dots,
\Z_{s_{d}},u) \mfa u>0 \eque
(see e.g., \cite[Exercise 6.6]{billingsley99convergence}).

By \eqref{eq:what-pi} the integrands on the right hand side of \eqref{aadvark2} converge pointwise under appropriate normalization.
That is, when  $A<1, C<1$ we have
\begin{equation}
  \label{eq:aadvark2a}
\limn\frac{n^{3/2}}{\mathfrak c_1}  \esp_{u}G_n\pp{\what Y_{s_1}\topp n,\dots,\what Y_{s_{d}}\topp n,u} \what \pi\topp n_0(u)=   \esp_u G(\Z_{s_1},\dots,\Z_{s_{d}},u)u^{1/2}
\end{equation}
and when $A=1,C<1$ we have
\begin{equation}
  \label{eq:aadvark2b}
\limn\frac{n^{1/2}}{\mathfrak c_2} \esp_{u}G_n\pp{\what Y_{s_1}\topp n,\dots,\what Y_{s_{d}}\topp n,u} \what
\pi\topp n_0(u)= \esp_u G(\Z_{s_1},\dots,\Z_{s_{d}},u)\frac1{u^{1/2}}.
\end{equation}
To conclude the proof we now apply
the
dominated convergence theorem.
We see that, if $x_d<1$,~\eqref{eq:Gn_bd} yields $G_n(\vv u) \le c\exp(-(1-x_d)u_{d+1}/4)$. Recall upper bounds on $\what \pi\topp n_0$ in~\eqref{eq:pihat_bd}.
Therefore,
the functions of $u$ that appear on the left-hand side
of \eqref{eq:aadvark2a} and  \eqref{eq:aadvark2b} are bounded by %
the integrable functions %
 $c u^{1/2} \exp\left(-(1-x_d)u/4\right)$ and  $c u^{-1/2} \exp\left(-(1-x_d)u/4\right)$,  respectively, for $n$ large enough. This proves the case $x_d<1$.

 To prove the case $x_d=1$, we need to work a little harder, although the approach is very similar.
 Notice that when $x_d = 1$, both $G_n(\vv u)$ and its limit  $G(\vv u)$ do  not depend on the last coordinate $u_{d+1}$. %
 Therefore we introduce
\[
 G_n^*\pp{u_1,\dots,u_d}  =G_n\pp{u_1,\dots,u_d,0}\qmand
G^*\pp{u_1,\dots,u_d} =G\pp{u_1,\dots,u_d,0}
\]
(the choice of 0 in the last variable of $G_n$ and $G$ is irrelevant for the definitions),
and write $\esp_u^*(\cdot)$ as either $\esp(\cdot\mid \what Y\topp n_{s_d} = u)$ or $\esp(\cdot\mid\Z_{s_d} = u)$, depending on whether the conditional expectation is for $\what Y\topp n$ or $\Z$. By conditioning on the value of $\what Y_{s_d}\topp n$, we  write
\begin{align*}
\esp G_n\pp{\what Y\topp n_{s_1},\dots,\what Y\topp n_{s_{d+1}}} & \equiv \esp G_n^*\pp{\what Y\topp n_{s_1},\dots,\what Y\topp n_{s_{d}}}
\\
& = \int _0^\infty \esp_u^*
G_n^*
\pp{\what Y\topp n_{s_1},\dots,\what Y\topp n_{s_{d-1}},u} \what \pi\topp n_{s_d}(u)du.
\end{align*}
The same argument in the proof of
Proposition~\ref{prop:tangent} for the convergence in law of tangent process leads to
\[
\calL\pp{\pp{\what Y\topp n_{s_1},\dots,\what Y\topp n_{s_{d-1}}}\mmid \what Y\topp n_{s_d} = u} \weakto \calL\pp{(\Z_{s_1},\dots,\Z_{s_{d-1}})\mid \Z_{s_d} = u}.
\]
 Therefore, %
by the same argument for \eqref{eq:weak_convergence}, we have
\[
\limn
\esp^*_uG_n^*\pp{\what Y\topp n_{s_1},\dots,\what Y\topp n_{s_{d-1}},u} = \esp_u^*G^*\pp{\Z_{s_1},\dots,\Z_{s_{d-1}},u},
\] and instead
of~\eqref{eq:aadvark2a} and~\eqref{eq:aadvark2b}, we have, for $A<1,C<1$,
\begin{equation*}
\limn\frac{n^{3/2}}{\mathfrak c_1}  \esp_{u}^*G_n^*\pp{\what Y_{s_1}\topp n,\dots,\what Y_{s_{d-1}}\topp n,u} \what \pi\topp n_{s_d}(u)=   \esp_u^* G^*(\Z_{s_1},\dots,\Z_{s_{d-1}},u)u^{1/2}
\end{equation*}
and for $A=1,C<1$,
\begin{equation*}
\limn\frac{n^{1/2}}{\mathfrak c_2} \esp_{u}^*G_n^*\pp{\what Y_{s_1}\topp n,\dots,\what Y_{s_{d-1}}\topp n,u} \what \pi\topp n_{s_d}(u)= \esp_u^* G^*(\Z_{s_1},\dots,\Z_{s_{d-1}},u)\frac{u^{1/2}}{u+s_d^2}.
\end{equation*}
One can show that for $n$ large enough,  $\what \pi\topp n_{s_d}$ has the same upper bounds as $\what\pi\topp n_0$ in~\eqref{eq:pihat_bd} (independent of $s_d$), by applying~\eqref{eq:pi_bd} to $\what \pi\topp n_{s_d}(u) = \pi_{s_d}(1-u/(2n))/2n$.
By~\eqref{eq:Gn_bd}, for $n$ large enough, $G_n(\vv u) \le c\exp(-(x_d-x_{d-1})u_d/8)$
(we cannot use $(x_d-x_{d-1})/4$ as before, because of the rounding issue caused by $n_k = \floor {nx_k}$).
So the dominated convergence theorem applies, and we arrive at
\begin{multline*}
\esp G_n^*\pp{\what Y\topp n_{s_1},\dots,\what Y\topp n_{s_{d}}} \\
\sim\left\{
\begin{array}{ll}
\displaystyle\frac{\mathfrak c_1}{n^{3/2}}\int_{\R_+}\esp_u^* G^*(\Z_{s_1},\dots,\Z_{s_{d-1}},u)u^{1/2}du & A<1, C<1\\ \\
\displaystyle\frac{\mathfrak c_2}{n^{1/2}}\int_{\R_+}\esp_u^* G^*(\Z_{s_1},\dots,\Z_{s_{d-1}},u)\frac{u^{1/2}}{u+s_d^2}du & A=1, C<1.
\end{array}
\right.
\end{multline*}
These are however not the same expressions as desired yet, and we
need to rewrite them. %
For the case $A<1, C<1$, since $\Z$ is stationary with respect to the distribution $u^{1/2}du$, we have (recalling that $\esp _u(\cdot) = \esp(\cdot\mid \Z_{s_{d+1}} = u)$)
\[
\int_{\R_+}\esp_u^* G^*(\Z_{s_1},\dots,\Z_{s_{d-1}},u)u^{1/2}du = \int_{\R_+}\esp_u G^*(\Z_{s_1},\dots,\Z_{s_{d-1}},   \Z_{s_d}
)u^{1/2}du.
\]
For the case $A=1, C<1$, because of the fact that \equh\label{eq:Z_A=1} \int_0^\infty \frac1{x^{1/2}} \qqq_{0,t}(x,y)dx =
\frac{\sqrt y}{y+t^2}, \eque
 (recall \eqref{eq:Z_pdf})
 we have
\[
\int_{\R_+}\esp_u^* G^*(\Z_{s_1},\dots,\Z_{s_{d-1}},u)\frac{u^{1/2}}{u+s_d^2}du = \int_{\R_+}\esp_u G^*(\Z_{s_1},\dots,\Z_{s_{d-1}},\Z_{s_{d}})\frac1{u^{1/2}}du.
\]
So to complete the proof it remains to show~\eqref{eq:Z_A=1}.
For this purpose, introduce change of variables $x = u^2, y = v^2$, and %
we get elementary integrals
\begin{align*}
\int_0^\infty \frac1{x^{1/2}}\qqq_{0,t}(x,y)dx& = 2\int_0^\infty \qqq_{0,t}(u^2,v^2)du \\
& = \int_0^\infty \frac1{\pi u}\pp{\frac t{t^2+(u-v)^2} - \frac t{t^2+(u+v)^2}}du\\
   &=\frac{t}{2
   \pi  \left(t^2+v^2\right)}\left.\log \frac{t^2+(u+v)^2}{t^2+(u-v)^2}\right|_{u=0}^{u=\infty} \\
   &\quad +\frac{v}{
   \pi  \left(t^2+v^2\right)}\left.\left(\arctan \left(\frac{u-v}{t}\right)+\arctan  \left(\frac{u+v}{t}\right)\right)\right|_{u=0}^{u=\infty}
   \\&=\frac{v}{t^2+v^2}=\frac{\sqrt{y}}{y+t^2}.
\end{align*}
\end{proof}

\begin{proof}[Proof of Theorem~\ref{T1''}]
By Lemma~\ref{lem:limit_Zn} and Proposition~\ref{prop:limit_Gn}, we have
\begin{equation}
  \label{lim-phi}
  \limn \varphi_{\vv x,n}\pp{\frac{\vv c}{\sqrt n}} = \left\{
\begin{array}{ll}
\displaystyle \frac1{4\sqrt{\pi}} \int_{\R_+}\esp_u G(\Z_{s_1},\dots,\Z_{s_{d}},u)u^{1/2} du & A<1,C<1\\ \\
\displaystyle \frac1{2\sqrt{\pi}} \int_{\R_+}\esp_u G(\Z_{s_1},\dots,\Z_{s_{d}},u)\frac{du}{u^{1/2}}  & A=1,C<1.
\end{array}
\right.
\end{equation}

We first prove the case of $A<1,C<1$.
Observe that %
\begin{multline}\label{eq:EG}
\int_{\R_+}\esp_u G(\Z_{s_1},\dots,\Z_{s_{d}},u)u^{1/2} du \\
= \exp\pp{\frac14\summ k1{d+1}s_k^2(x_k-x_{k-1})}\times\int_{\R_+} \esp_u%
{\exp\pp{-\frac14\summ k1{d+1}\Z_{s_k}(x_k-x_{k-1})}}u^{1/2}du.
\end{multline}
The first exponential function on the right-hand side above corresponds to the Laplace transform of the scaled Brownian motion %
$8^{-1/2}\B$ (see~\eqref{eq:Brownian} %
and recall that in \eqref{eq:Brownian} we used $s_k = c_k+\cdots+c_d$, but we have been using $s_k = (c_k+\cdots+c_d)/2$ since \eqref{eq:s_k}).
The integral on the right-hand side of~\eqref{eq:EG}  corresponds to the Laplace transform of a Brownian excursion.
Indeed,  by self-similarity of the process $\Z$,  it equals
\begin{align*}
\int_{\R_+}u^{1/2}du&  \esp_{u}\exp\pp{-\frac12\summ k1{d+1}\frac12\Z_{s_k}(x_{k}- x_{k-1})}\\
& =\int_{\R_+}u^{1/2}du  \esp_{u/2}\exp\pp{-\frac12\summ k1{d+1}\Z_{s_k/\sqrt{2}}(x_{k}- x_{k-1})}\\
& =\sqrt 8 \int_{\R_+}u^{1/2}du \esp_{u}\exp\pp{-\frac12\summ k1{d+1}\Z_{s_k/\sqrt{2}}(x_{k}- x_{k-1})}.
\end{align*}
By the duality expression in~\eqref{eq:ex}, this becomes
\[
 4\sqrt{\pi}\,\esp\exp\pp{-\summ k1d\frac{ s_k- s_{k+1}}{\sqrt 2}\B^{ex}_{ x_{k}}} =  4\sqrt{\pi}\, \esp\exp\pp{-\summ k1d%
 \frac{c_k}{2\sqrt 2}
 \B^{ex}_{ x_{k}}}.
\]
Combining all the
identities
 together we have proved~\eqref{eq:Laplace_convergence} for $A<1, C<1$.\medskip

Now we prove the case of $A=1,C<1$. This time we have %
\begin{multline}\label{eq:EG2}
\int_{\R_+^{d+1}}\esp_u G(\Z_{s_1},\dots,\Z_{s_{d}},u)\frac1{u^{1/2}} du \\
=\esp\exp\pp{-\summ k1d%
\frac{c_k}{2\sqrt 2}\B_{x_k}
}\times
\int_{\R_+}\frac1{u^{1/2}}du\esp_u\exp\pp{-\frac14\summ k1{d+1}\Z_{s_k}(x_k-x_{k-1})}.
\end{multline}
\begin{multline*}
\int_{\R_+ }\esp_u G(\Z_{s_1},\dots,\Z_{s_{d}},u)\frac{du}{u^{1/2}}  \\
=\esp\exp\pp{-\summ k1d%
\frac{c_k}{2\sqrt 2}\B_{x_k}
}\times
\int_{\R_+}%
\esp_u%
{\exp\pp{-\frac14\summ k1{d+1}\Z_{s_k}(x_k-x_{k-1})}} \frac{du}{u^{1/2}}.
\end{multline*}
Again by self-similarity and duality~\eqref{eq:me}, we rewrite  the  integral on the right-hand side above as
\begin{multline*}
 \sqrt 2\int_{\R_+} %
 \esp_u%
 {\exp\pp{-\frac12\summ k1{d+1}\Z_{s_k/\sqrt 2}( x_{k}- x_{k-1})}}\frac{du}{u^{1/2}}\\
 = 2\sqrt \pi\, \esp\exp\pp{-\summ k1d\frac{s_k-s_{k+1}}{\sqrt 2}\B_{x_k}^{me}}
  = 2\sqrt \pi\, \esp\exp\pp{-\summ k1d\frac{c_k}{2\sqrt 2}\B_{x_k}^{me}}.
\end{multline*}
This completes the proof for the case $A=1$ and $C<1$.
\end{proof}

\subsection{Proof for the case $A<1, C=1$}\label{sec:A<1C=1}
This case  of Theorem \ref{T1}  follows from the case $A=1,C<1$
again by the particle-hole duality, which we have explained in Section~\ref{sec:particle_hole}.
Write $\eps_j=1-\tau_{n-j+1}$, and define
$$\what h_n(x)=\summ j1{\floor{nx}}\pp{\varepsilon_j-\frac12}.$$
Similarly as in~\eqref{eq:particle_hole}, this time we have
\[
h_n(x) + \pp{\what h_n(1) - \what h_n(1-x)} = \left\{
\begin{array}{ll}
0 & nx =\floor{nx}\\ \\
\displaystyle -\pp{\varepsilon_{\floor{n(1-x)}+1}-\frac12} & nx\ne \floor{nx}
\end{array}
\right.,
\]
and $\{\what h_n(x)\}_{x\in[0,1]}$ with respect to
$\pi_n^{A,B,C,D}$ has the same law as $\{h_n(x)\}_{x\in[0,1]}$ with respect to
$\pi_n^{C,D,A,B}$. Therefore,  finite-dimensional distributions of the
process
$n^{-1/2}\ccbb{h_n(x)}_{x\in[0,1]} $  have the same limit as the finite-dimensional distributions of the process $n^{-1/2} \sccbb{ {\what h_n(1-x) - \what h_n(1)}}_{x\in[0,1]}$.
 We get  %
\begin{eqnarray*}
\frac1{\sqrt n}\ccbb{h_n(x)}_{x\in[0,1]} & \fddto  & \frac1{2\sqrt 2}\ccbb{ \B_{1-x} + \B_{1-x}^{me}
-(\B_1 + \B_1^{me} )}_{x\in[0,1]} \\
& \eqd & \frac1{2\sqrt 2}\ccbb{\B_x + \B_{1-x}^{me} - \B_{1}^{me}}_{x\in[0,1]}.
\end{eqnarray*}
This completes the proof of Theorem~\ref{T1}.

\subsection*{Acknowledgement}
 The authors benefited from references provided by G\'erard Letac, Jacek Weso\l owski, and Przemys\l aw Matu\l a.
  WB's research was supported in part by the Charles Phelps Taft Research Center at the University of Cincinnati.
YW's research was supported in part by NSA grant H98230-16-1-0322 and Army Research Laboratory grant W911NF-17-1-0006.

\appendix
\section{Weak Convergence from convergence of Laplace transforms}\label{appendix:Laplace}
 It is well known that convergence of Laplace transforms in a neighborhood of $\vv{0}\in\R^d$ implies weak convergence, but it is less known
under what conditions  convergence on an open set away from the origin suffices.
 \citet[Section 5.14, page 378, (5.14.8)]{hoffmann1994probability} and  %
 Mukherjea, Rao, and Suen
  \cite[Theorem 2]{mukherjea2006note} independently discovered the pertinent result
   in the univariate case, and the argument in \cite{mukherjea2006note} generalizes to the multivariate setting,
   compare also \cite[Theorem 2.1]{farrell2006techniques}.

Let   $\vv{X}\topp n=(X_1\topp n,X_2\topp n,\dots,X_d\topp n)$   be a sequence of  random vectors with Laplace transform
 \[
 L_n(\vv{z})=L_n(z_1,\dots,z_d)=%
\esp \exp\pp{\sum_{j=1}^d z_j X_j\topp n}.
\]
\begin{theorem}\label{Thm-Laplace}
Suppose that $L_n(\vv{z})$ are finite and converge pointwise to a function $L(\vv{z})$ for all $\vv{z}$ from an open set in $\R^d$.
If on this open set $L(\vv{z})$ is the Laplace transform of a random variable  $\vv{Y}=(Y_1,\dots,Y_d)$,  then
$\vv{X}\topp n$ converges in distribution to $\vv{Y}$.
\end{theorem}

\arxiv{
\begin{proof} The proof is a  modification of \cite[proof of Theorem 2]{mukherjea2006note}.  Denote by $P_n$ the law of
$\vv X\topp n$
 and by $P_\infty$ the law of $\vv{Y}$.
Choose $\eps>0$ and $\vv{c} \in\R^d$   such that $L_n(\vv{z})\to  L(\vv{z})$ for all $\vv{z}$ from the open ball  $\|\vv{z}-\vv{c}\|<\eps$.
Set  $C_n=L_n(\vv{c})$, $C_\infty=L(\vv{c})$ and consider probability measures
$$Q_n(d\vv{x})=\frac{1}{C_n}e^{\vv{x}\cdot\vv{c}}P_n(d\vv{x}),$$
with $n=\infty$ standing for the limiting measure.
Due to our choice of %
$\vv c$,
$$\int_{\R^d} e^{\vv{t}\cdot \vv{x}}Q_n(d\vv{x})= \frac{L_n(\vv{t}+\vv{c})}{C_n} \to
  \frac{L(\vv{t}+\vv{c})}{C_\infty}=\int_{\R^d} e^{\vv{t}\cdot \vv{x}}Q_\infty(d\vv{x})$$
for all $\|\vv{t}\|<\eps$. By the ``usual form" of the Laplace criterion (for example, by the %
Cram\'er--Wold device and %
\citet[Theorem 3]{curtiss1942note}), this
implies weak convergence of $Q_n$ to $Q_\infty$, i.e. for
every bounded continuous function $h$,
$$\int_{\R^d} h(\vv{x})Q_n(d\vv{x})\to \int_{\R^d} h(\vv{x})Q_\infty(d\vv{x}).$$
Mimicking \cite{mukherjea2006note} we take  $0\leq h \leq 1$, $\eps>0$ and note that $f(\vv{x})=e^{\vv{x}\cdot\vv{c}}$ is  strictly positive, so
$$
h_\eps(\vv{x})= \frac{h(\vv x) f(\vv{x})}{f(\vv{x})+\eps}\nearrow h(\vv x) \mbox{ %
for all $\vv x$
as $\eps\searrow 0$}.
$$
Therefore,
\begin{multline*}
\liminf_{n\to\infty} \int h(\vv{x})P_n(d\vv{x})\geq \liminf_{n\to\infty}\int h_\eps(\vv{x})P_n(d\vv{x})
\\=
\lim_{n\to\infty}C_n\int \frac{h(\vv x)  }{f(\vv{x})+\eps} Q_n(d\vv {x})
=C_\infty\int \frac{h(\vv x)  }{f(\vv{x})+\eps} Q_\infty(d\vv {x})\\= \int h_\eps(\vv{x})P_\infty(d\vv{x}).
\end{multline*}
Taking the limit as $\eps\to0$, we get
$$\liminf_{n\to\infty} \int h(\vv{x})P_n(d\vv{x})\geq  \int h(\vv{x})P_\infty(d\vv{x}).$$
Applying the above to $1-h$, we see that
$$\lim_{n\to\infty} \int h(\vv{x})P_n(d\vv{x})= \int h(\vv{x})P_\infty(d\vv{x})$$
for all continuous functions $0\leq h\leq 1$, which ends the proof.
 \end{proof}
 }

\bibliographystyle{apalike}
\bibliography{references}
\end{document}